\documentclass[11pt]{article}
\usepackage{amssymb,amsmath,amsthm,amsfonts, commath, mathrsfs, mathtools, secdot}
\usepackage[T1]{fontenc}		
\usepackage[a4paper, margin=0.97in]{geometry}

\usepackage{bbm}
\usepackage{extarrows}
\usepackage{enumerate}
\usepackage{multicol}	
\usepackage{fge}
\usepackage{oldgerm}
\usepackage{hyperref} 
\usepackage{paralist} 
\usepackage{accents}
\usepackage{pifont}
\usepackage{authblk}

\makeatletter
\newcommand{\myitem}[1]{%
\item[#1]\protected@edef\@currentlabel{#1}%
}
\makeatother

\usepackage{tikz}
\usetikzlibrary{datavisualization}
\usetikzlibrary{datavisualization.formats.functions}
\usepackage{pgfplots}
\usepackage{tkz-fct}
\pgfplotsset{every axis/.append style={
    axis x line=middle,    
    axis y line=middle,    
    axis line style={<->}, 
    xlabel={$x$},          
    ylabel={$y$},          
    },
    cmhplot/.style={color=blue,mark=none,line width=1pt,<->},
    soldot/.style={color=blue,only marks,mark=*},
    holdot/.style={color=blue,fill=white,only marks,mark=*},
}
    
\tikzset{>=stealth}

\usepackage{blindtext}
\usepackage{caption}
\usepackage{csquotes}

\sectiondot{subsection}
\usepackage{xcolor}

\usepackage{color}

\newtheoremstyle{Assump}%
  {3pt}
  {3pt}
  {\itshape}
  {}
  {\bfseries}
  {.}
  {.5em}
  {\thmname{#1} \thmnumber{#2} \thmnote{\normalfont#3}}

\newtheorem{theorem}{Theorem}[section]
\newtheorem{lemma}[theorem]{Lemma}
\newtheorem{proposition}[theorem]{Proposition}

\theoremstyle{definition}\newtheorem{example}[theorem]{Example}
\theoremstyle{definition}
\theoremstyle{definition}
\theoremstyle{definition}

\theoremstyle{definition}\newtheorem{criterion}[theorem]{Criterion}

\newtheorem{innercustomthm}{Condition}
\newenvironment{Condition}[1]
  {\renewcommand\theinnercustomthm{#1}\innercustomthm}
  {\endinnercustomthm}

\newtheorem{innercustomthmm}{Conjecture}
\newenvironment{Conjecture}[1]
  {\renewcommand\theinnercustomthmm{#1}\innercustomthmm}
  {\endinnercustomthmm}

\numberwithin{equation}{section}

\newcommand{\N}{\mathbb{N}}
\newcommand{\R}{\mathbb{R}}

\newcommand{\D}{{\bf D}}											

\newcommand{\F}{\mathbb{F}}											
\newcommand*\diff{\mathop{}\!\mathrm{d}}							
\newcommand{\E}{\mathbb{E}}											
\newcommand{\Pro}{\mathbb{P}}										
\newcommand{\ind}{\operatorname{\mathbf{1}}}						
\newcommand{\Id}{\operatorname{Id}}									
\renewcommand{\sup}[1]{\underset{#1}{\operatorname{sup}}\;}			
\renewcommand{\inf}[1]{\underset{#1}{\operatorname{inf}}\;}			
\renewcommand{\setminus}{\mathbin{\fgebackslash}}					
\newcommand{\union}[2]{\overset{#2}{\bigcup\limits_{#1}}\;}			
				
\newcommand{\conv}[2]{\; \xrightarrow[ {#1}]{#2} \; }			      
   
\newcommand{\dJ}{\diff_{\operatorname{\operatorname{J}_1}}}			
\newcommand{\dM}{\diff_{\operatorname{\operatorname{M}_1}}}			

\title{Relative compactness of Itô integrals on the $M_1$ Skorokhod space with identification of limits}
\author{}
\date{}

\author{Fabrice Wunderlich \thanks{Department of Mathematics, University of Luxembourg. Email: \texttt{fabrice.wunderlich@uni.lu}}
\thanks{Department of Statistics, London School of Economics}}

\begin{document}
\maketitle

\begin{abstract} We establish the weak relative compactness of sequences of Itô integrals with respect to Skorokhod's functional $M_1$ topology, under sharp general conditions. Moreover, we are able to explicitly characterise the form of the limit points of all convergent subsequences. These results close a longstanding gap in the literature, where weak relative compactness had previously only been shown under the significantly more restrictive $J_1$ topology, or the $S$ topology which is too coarse to preserve continuity for most common functionals.
\end{abstract}



\section{Introduction} \label{sec:intro}

We consider the problem of weak relative compactness of sequences of stochastic Itô integrals with respect to Skorokhod's $M_1$ topology. Let $\D_{\R^d}[0,\infty)$ be the space of càdlàg paths $x:[0,\infty)\to \R^d$, and let $\dJ$, $\dM$ be the classical metrics inducing the $J_1$ \cite[(12.13)]{billingsley} and $M_1$ topology \cite[Sec.~12.3.1]{whitt} respectively. Given a sequence $\{X^n\}$ of semimartingales on filtered probability spaces $(\Omega^n,\mathcal{F}^n, \mathbbm{F}^n,\Pro^n)$, and a sequence $\{H^n\}$ of $\mathbbm{F}^n$-adapted càdlàg processes, we aim to identify sharp general conditions that ensure the weak relative compactness of the sequence
\begin{align} \Big\{\int_0^\bullet \, H^n_{s-} \, \diff X^n_s\Big\}\label{intro:integral_sequence_in_question} \end{align}
on $(\D_{\R^d}[0,\infty),\dM)$, assuming that there is weak convergence $(H^n,X^n) \Rightarrow (H^0,X^0)$ on the product space $(\D_{\R^d}[0,\infty),\dM) \times (\D_{\R^d}[0,\infty),\dM)$. Moreover, we shall explicitly characterise the form of limit points of any weakly convergent subsequences.

Such results will close a longstanding gap in the literature. They are of natural theoretical interest, as they provide a rigorous description of the limiting behaviour of stochastic integrals under various types of approximations of the càdlàg integrand and integrator. Most notably this includes continuous approximations such as classical convolution-based smoothing, which are \emph{not} covered by the $J_1$ framework. Indeed, the $M_1$ topology---in stark contrast to the more traditionally used $J_1$ topology---allows for jumps to be approximated by continuous components or through staggered smaller jumps, while still preserving continuity of most common functionals. Therefore, the results in this article advance the understanding of weak convergence in the more versatile $M_1$ topology, and provide a natural extension of \cite{andreasfabrice_theorypaper}. The proofs are technically involved and require novel approaches.


\subsection{Classical approaches to weak continuity of Itô integrals}
Classical approaches, which pursued 'full' weak convergence of the sequence \eqref{intro:integral_sequence_in_question}, typically start by imposing a critical control on the behaviour of the integrators $X^n$ (see e.g. the UT condition \cite[p.112]{jakubowskimeminpages}, the UCV condition \cite[C2.2(i)]{kurtzprotter}, or GD in \cite{andreasfabrice_theorypaper}). Several fundamental examples such as \cite[§6b, p.~38]{shiryaev} and \cite[Ex.~3.17]{andreasfabrice_theorypaper} demonstrate that such control can generally not be omitted without risking the stochastic unboundedness of the integral processes, even when the individual processes converge uniformly. In the sequel, we will therefore assume such control, that is that the integrators $X^n$ decompose as
\[
	X^n = M^n + A^n,\quad M^n  \text{ local martingales},\quad  A^n \text{ finite variation processes},
	\]
	such that, for every $t>0$, we have 
	\begin{align}\label{eq:Mn_An_condition}\tag{GD}
		\lim_{R\rightarrow \infty}  \limsup_{n\rightarrow\infty } \, \mathbb{P}^n\bigl(\text{TV}_{[0,t]}(A^n)>R\bigr)=0 \quad \; \; \text{and}  \; \; \quad
		\limsup_{n\rightarrow \infty} \,  \mathbb{E}^n\bigl[\, |\Delta M^n_{t \land \tau^n_c}|\, \bigr] <\infty,
	\end{align}
	for all $c>0$, where $\tau^n_c:= \operatorname{inf}\{s >0: |M^n|^*_s\ge c \}$. Here, we have denoted $|Z|^*_t:= \operatorname{sup}_{0\le s\le t} |Z_s|$ the running supremum of a stochastic process $Z$ over the time interval $[0,t]$, $\text{TV}_{[0,t]}(Z)$ its total variation on $[0,t]$, and $\Delta Z_t:= Z_t - Z_{t-}$ the jump of $Z$ at time $t$.

If one seeks to obtain 'full' weak convergence of the sequence \eqref{intro:integral_sequence_in_question}, an additional assumption beyond \eqref{eq:Mn_An_condition} is required, acting on the interplay between the integrands and integrators beyond adaptedness and joint weak convergence (see e.g. \cite[(AVCI) \& Thm.~3.6]{andreasfabrice_theorypaper}). Originally introduced by \cite[Eq.~(6)]{jakubowski}, the condition rules out that---asymptotically---there is a significant increment of the integrands before one of the integrators. This prevents the approximating integrals from accruing a mass surplus or deficit that does not translate to the limiting integral. More formally, denote $a\wedge b:= \operatorname{min}\{a,b\}$ and $x^{(i)}$ the $i$-th coordinate of $x$. Define a function $\hat{w}^T_\delta: \D_{\R^d}[0,\infty) \times \D_{\R^d}[0,\infty) \to \R_+$ of the largest consecutive increment within a $\delta$-period of time on $[0,T]$, namely
\begin{align}
	\hat{w}^T_{\delta}(x,y)  := \operatorname{sup}\!\left\{  |x^{(i)}(s)- x^{(i)}(t)| \wedge |y^{(i)}(t)-y^{(i)}(u)|  :   s < t  < u \le (s+\delta), \, 1\le i\le d    \right\}, \label{defi:AVCI_modulus_of_continuity}
\end{align}
where the supremum is restricted to $0\leq s, u\leq T$. Then, the additional condition on the interplay can be stated as
	\begin{align} 
		\lim\limits_{\delta \downarrow 0} \,  \limsup\limits_{n\rightarrow \infty} \,  \Pro^n\bigl( \, \hat w_{\delta}^T(H^n, \, X^n) \; > \; \gamma \, \bigr) \; = \; 0. \tag{AVCI} \label{eq:oscillcond}
	\end{align}
In \cite[Prop.~3.8]{andreasfabrice_theorypaper} it has been shown that, whenever the limiting quantities $H^0$, $X^0$ almost surely have no common discontinuities, or if the pairs $(H^n,X^n)$ converge on the strong $J_1$ space $(\D_{\R^{2d}}[0,\infty),\dJ)$ as opposed to on $(\D_{\R^{d}}[0,\infty),\dJ)\times (\D_{\R^{d}}[0,\infty),\dJ)$ (which quintessentially requires the integrands to eventually jump at the same time as the integrators), then \eqref{eq:oscillcond} is satisfied. In fact, the most influential seminal papers \cite{jakubowskimeminpages, kurtzprotter} for the $J_1$ topology built their weak convergence results for stochastic integrals around the latter sufficient criterion. We shall stress that it is significantly more intricate to show tightness (and thus, by extension via Prokhorov's theorem, weak relative compactness) of the pairs $(H^n,X^n)$ on the strong space $\D_{\R^{2d}}[0,\infty)$ than on the product space $(\D_{\R^{d}}[0,\infty))^2$ (cf.~\cite[Rmk.~3.10]{andreasfabrice_theorypaper}). While both criteria provide a simple lever in many interesting cases, beyond these instances it quickly becomes either unclear how to practically verify \eqref{eq:oscillcond} or the condition might simply not be satisfied. 
Example \ref{ex:simple_example_tightness_without_AVCI} suggests that, even if \eqref{eq:oscillcond} fails, one may still expect at least the weak relative compactness of the integrals.

\begin{example} \label{ex:simple_example_tightness_without_AVCI}
    Consider first the deterministic sequences defined by $x_n=\ind_{[1-1/n,\, \infty)}$, as well as $h_n$ equal to $h_n^{'}:=\ind_{[0,\,1-2/n)}+3\ind_{[1-2/n,\, \infty)}$ if $n$ is even, and otherwise equal to the function $h_n^{''}:=\ind_{[0,\,1+1/n)}+3\ind_{[1+1/n,\, \infty)}$. Obviously, $h_n \to \ind_{[0,\, 1)}+3\ind_{[1,\,\infty)}=:h_0$, $x_n\to \ind_{[1,\,\infty)}=:x_0$ in the $J_1$ topology but due to the preceding jump of $h_n$ before the one of $x_n$ for even $n$, we have
    $$ \int_0^\bullet h_n(s-)\diff x_n(s)\; = \; \begin{cases} 3\ind_{[1-1/n,\, \infty)}, \quad & \text{if $n$ even} \\ \ind_{[1-1/n,\,\infty)}, \quad & \text{if $n$ odd}  \end{cases}$$
    and therefore we merely obtain relative compactness, with $J_1$ convergence of the subsequences 
    $$\int_0^\bullet h_{2n}(s-)\diff x_{2n}(s) \, \to \, 3\ind_{[1,\,\infty)} \, = \, \int_0^\bullet h_0(s-) \diff x_0(s) + \Delta h_0(1) \Delta x_0(1) \ind_{\{1\le \bullet\}}$$ as well as $$\int_0^\bullet h_{2n+1}(s-)\diff x_{2n+1}(s) \, \to \, \ind_{[1,\,\infty)} \, = \, \int_0^\bullet h_0(s-) \diff x_0(s).$$
    Extending the above example in a way that introduces stochasticity, consider now Bernoulli random variables $\xi_n$ and 
    $$ H^n \; := \; \xi_n h_n^{'} \; + \; (1-\xi_n)h_n^{''} \; \to \; h_0 \; =: \; H^0$$
    almost surely in $J_1$, while leaving the $X^n:=x_n \to x_0=: X^0$ unchanged as above. The $\xi_n$ trivially being tight, by Prokhorov's Theorem there exists a subsequence $\{\xi_{n_k}\}$ of the latter converging in distribution to some Bernoulli $\xi$, and thus 
    \begin{align*} \int_0^\bullet H^{n_k}_{s-} \diff X^{n_k}_s \; &= \;\int_0^\bullet h_{n_k}^{''} \diff X^{n_k}_s \, + \, \xi_{n_k} \int_0^\bullet \Delta h^{'}_{n_k}(1-2/n_k)\, \ind_{[1-2/n_k,\, 1+1/n_k)}(s-) \diff X^{n_k}_s  \\[1ex]
    &= \; \ind_{[1-1/n_k,\, \infty)} \, + \, \xi_{n_k} \, \Delta h^{'}_{n_k}(1-2/n_k) \Delta x_{n_k}(1-1/n_k) \ind_{[1-1/n_k, \, \infty)} \\[1ex]
    &\Rightarrow \; \ind_{[1,\, \infty)} \, + \, \xi \, \Delta h_0(1) \Delta x_0(1) \ind_{[1,\, \infty)} \, = \, \int_0^\bullet H^0_{s-}\diff X^0_s \, + \, \xi \, \Delta H^0_1 \Delta X^0_1\ind_{[1,\, \infty)}.
    \end{align*}
\end{example}

\subsection{A gap in the literature for weak relative compactness}
Indeed, while a lack of \eqref{eq:oscillcond} generally jeopardises 'full' functional weak convergence, Kurtz \& Protter  \cite[pp.~1051-1054]{kurtzprotter} were able to show that, for $J_1$ convergent semimartingale integrands and integrators $(H^n,X^n)\Rightarrow (H^0,X^0)$ on $(\D_{\R^d}[0,\infty),\dJ)\times (\D_{\R^d}[0,\infty),\dJ)$ where the integrators $X^n$ have UCV (which, in fact, is equivalent to \eqref{eq:Mn_An_condition} since $X^n \Rightarrow X^0$; see \cite[App.~C]{andreasfabrice_theorypaper}), the sequence $\{(X^n, \int_0^\bullet H^n_{s-} \diff X^n_s)\}$ is weakly relatively compact on $(\D_{\R^{2d}}[0,\infty), \dJ)$. Moreover, they identified the limit points of converging subsequences as \\[-3ex]
\begin{align} \Big(X^0,\, \int_0^\bullet H^0_{s-}\, \diff X^0_s \, + \, \sum_{\beta_i \le \bullet} \, \Delta H^0_{\beta_i} \Delta X^0_{\beta_i}\Big) \label{kurtz_protter_J1_J1} \\[-5ex] \notag
\end{align}
where $\{\beta_i\}$ is some subset of the random jump times of $(H^0,X^0)$. Clearly, the fact that weak relative compactness in $J_1$ is guaranteed despite a potential absence of \eqref{eq:oscillcond} is immensely useful. Yet, it becomes more and more clear that the $J_1$ topology presents substantial rigidity for many modern modeling purposes as it requires both the jump times and jump sizes of a convergent sequence to approximately match those of the limit.

If we leave the realm of regularity offered by Skorokhod's $J_1$ and $M_1$ topologies, \cite{jakubowski2} introduced a significantly coarser 'ultraweak' topology, the so-called $S$ topology, which is closely linked to the conditions F1--F3 in \cite[Prop.~2.16]{andreasfabrice_theorypaper}. As shown by \cite{stricker}, the UT condition, popularised in \cite{jakubowskimeminpages}, implies tightness with respect to the $S$ topology (\cite[Thm.~4.1]{jakubowski2}). Therefore, since UT propagates from the integrators to the integral, one obtains $S$ tightness with the fairly minimal assumption of UT. As described in \cite[App.~C]{andreasfabrice_theorypaper}, just like the UCV condition the UT condition is equivalent to \eqref{eq:Mn_An_condition} since $X^n \Rightarrow X^0$, and thus $S$ tightness carries over to our framework of \eqref{eq:Mn_An_condition} instead of UT. Although the $S$ space is not metrisable, and  hence we cannot directly apply the Prokhorov Theorem in order to deduce weak relative compactness from tightness, \cite[Thm.~3.4]{jakubowski2} showed that this very implication is nonetheless true for the S topology. While, in certain specific cases, $S$ tightness might thus turn out to be practical, we shall note that many useful continuous functionals in the classical Skorokhod spaces, such as, e.g., the running supremum or time evaluations, are generally \emph{nowhere} continuous on the $S$ space. This critically limits the number of potential use cases for weak limit theorems with respect to the $S$ topology.

In contrast, Skorokhod’s $M_1$ topology bridges the gap between the rigidity of the $J_1$ topology and the looseness of the $S$ topology. Indeed, it enables a flexible approximation of limiting jumps through combinations of staggered smaller jumps or continuous pieces, as long as these almost monotonically morph into the limiting discontinuity. For both an excellent introduction to and an advanced study of the $M_1$ topology we refer to Whitt's canonical monograph \cite{whitt}. At the same time, convergence in the $M_1$ topology is sufficiently robust, with many
natural functionals (such as, e.g., the running supremum and time coordinates evaluated at continuity times of the limit) preserving convergence. Moreover, compactness in the $M_1$ topology is very often easier to achieve, especially when monotonic components are present. For these reasons, the $M_1$ topology increasingly becomes a standard tool for the study of functional convergence on the space of càdlàg paths. It has been used for a wide range of applications, playing a role in addressing problems related to financial mathematics \cite{Becherer_M1}, queueing and network theory \cite{PangWhitt, Mikosch, ramanan_M1}, the analysis of physical systems \cite{Jung_Zhang, Melbourne_Varandas, Dembo_LiCheng}, the study of anomalous diffusion \cite{scalas, andreasfabrice_CTRWpaper}, the Bouchaud trap model \cite{arous_btm, croydonM1}, as well as to the study of mean field interaction models  \cite{Fu_Horst, DIRT_SPA, mercy, Nadtochiy_Skholni2}.

\subsection{Overview of the main results}
In this article, we establish the weak relative compactness of Itô integrals on the $M_1$ Skorokhod space (or $J_1$ space if the integrators converge in $J_1$), whenever (1) the integrators satisfy \eqref{eq:Mn_An_condition}, (2) there is some essential jump-alignment condition satisfied by the limiting pair $(H^0,X^0)$ (Theorem \ref{thm:tightness_without_AVCI}). Notably, if only the integrands converge in $M_1$ and the integrators converge in $J_1$, then this jump-alignment condition is no longer required (Theorem \ref{cor:main_result_M1_J1_case}). The form of the limit points for convergent subsequences will be explicitly identified as the limiting stochastic integral $\int H^0_{s-}\diff X^0_s$ plus the sum of jump products of the integrand $H^0$ and integrator $X^0$, weighed according to random variables taking values in $[0,1]^d$. Furthermore, we show that if, contrary to \eqref{eq:oscillcond}, asymptotically all increments of the integrators occur after those of the integrands, then we are in fact able to recover again 'full' weak convergence. Here, the weights of all jump products in the correction term of the limit are equal to one (Proposition \ref{cor:integral_conv_without_AVCI}).


\section{Extending relative compactness to $M_1$} \label{sec:main_results}
The results on weak relative compactness in the setting of integrands and integrators converging with respect to either the $J_1$ topology or the 
$S$ topology, as described in Section \ref{sec:intro}, leave a fundamental gap in the literature for analogue statements within the framework of the $M_1$ topology. This gap will be addressed in the sequel. Whenever we have $(H^n,X^n) \Rightarrow (H^0,X^0)$ on $(\D_{\R^d}[0,\infty),\diff_{\tilde \rho}) \times (\D_{\R^d}[0,\infty), \diff_{\rho})$, with $\tilde\rho, \rho \in \{\text{$J_1$},\text{$M_1$}\}$, we will refer to this as the $\tilde\rho$--$\rho$ case.

\subsection{The $M_1$--$M_1$ case} \label{sec:main_results_M1_M1}
Under a certain jump-alignment condition imposed on the integrands, we will extend Kurtz \& Protter's result \eqref{kurtz_protter_J1_J1} on weak relative compactness in the $J_1$--$J_1$ case to our more general $M_1$--$M_1$ framework. The condition is sharp as one can see from the simple deterministic Example \ref{ex:simple_example_failure_of_convergence_without_regularity_condition_tightness_without_AVCI}, highlighting how jumps in the limiting integrands that entail an instantaneous sign reversal may generally jeopardise $M_1$ convergence of any subsequence. Therefore, in order to obtain weak relative compactness of Itô integrals under otherwise reasonably general assumptions, a condition to rule out such behaviour \emph{cannot} be omitted. Importantly, it highlights the 'price to pay' for the greater flexibility inherent to the $M_1$ framework.

\begin{example} \label{ex:simple_example_failure_of_convergence_without_regularity_condition_tightness_without_AVCI}
    Consider
    $$h_n:= 1 - 2\ind_{[1-2/n,\infty)}\quad \text{ and } \quad  x_n:= (1/2) \ind_{[1-3/n,\infty)} + (1/2) \ind_{[1-1/n,\infty)}$$
    and note that $h_n \to h_0:=1-2\ind_{[1,\infty)}$ in the $J_1$ topology while $x_n \to x_0:=\ind_{[1,\infty)}$ in $M_1$. The integral process $\int_0^\bullet h_n(s-) \diff x_n(s)= (1/2) \ind_{[1-3/n,\, 1-1/n)}$ is evidently not relatively compact in the $M_1$ topology (as illustrated in Figure \ref{fig:integral_J1_M1_without_AVCI_fail}).
    \begin{figure}[bp]
    \hspace{3cm}
	\begin{tikzpicture}[scale=0.8]
		\begin{axis}[
			xmin=0,xmax=5.2,
			ymin=0,ymax=0.6,
			xlabel = {$t$},
			ylabel = {{\color{blue} $\int_0^t h_n(s-)\diff x_n(s)$}},
			ytick = {0, 0.5},
			xtick = {0,3,4,5},
			xticklabels={$0$,$1 - \frac{3}{n}$,$1-\frac{1}{n}$,$1$},    axis x line=middle,
			axis y line=middle,
			every axis x label/.style={at={(ticklabel* cs:1.03)}, anchor=west},
			every axis y label/.style={at={(ticklabel* cs:1.03)},anchor=south},
			]
			\addplot[cmhplot,-,domain=0:3]{0};
			\addplot[cmhplot,-,domain=3:4]{0.5};
			\addplot[cmhplot,-,domain=4:5]{0};
			\addplot[holdot, fill=white]coordinates{(3,0)};
			\addplot[holdot, fill=white]coordinates{(4,0.5)};
			\addplot[soldot]coordinates{(3,0.5)};
			\addplot[soldot]coordinates{(4,0)};
		\end{axis}
	\end{tikzpicture}
		\caption{Loss of relative compactness for $J_1$--$M_1$ integrals due to a change of sign of the integrands.}
		\label{fig:integral_J1_M1_without_AVCI_fail}
	\end{figure}
\end{example}
In Condition \ref{ass:R1}, we formalise a suitable limiting jump-alignment ensuring the desired extension to the $M_1$--$M_1$ case.

\begin{Condition}{R1} \label{ass:R1}
        For the limiting processes $H^0,X^0$ it holds coordinatewise that for any $s>0$,
    $$ H^0_{s-}\Delta H^0_{s} |\Delta X^0_s| \; \ge \; 0 \qquad \text{$\Pro^0$-a.s.\,.}$$
\end{Condition}
As one can see, Condition \ref{ass:R1} requires more of the limiting pair $(H^0,X^0)$ than simply ruling out instantaneous sign reversals of $H^0$ at common discontinuities. In fact, we suspect that Condition \ref{ass:R1} could be replaced by a weaker condition in the sequel (see \hyperref[conj]{Conjecture} below).

\begin{Condition}{R2} \label{ass:R2}
    For any strictly decreasing null sequence $\{a_k\} \subseteq (0,\infty)$, it holds that for all $T>0$, $\varepsilon>0$,
    \begin{align*}
        \Pro^0 \Big(  \sum_{s\le T} |\Delta H^0_s \Delta X^0_s|  \ind_{\{0 \, < \,|\Delta H^0_s|\, \le \, a_k\}} \, > \, \varepsilon\Big) \; \conv{k\to \infty} \; \; 0.
    \end{align*}
\end{Condition}

We shall effectively view Condition \ref{ass:R2} as a means to warrant the existence of the process of summed up jump products $\sum_{s\le \bullet} \Delta H^0_s \Delta X^0_s$, which will serve as the fundamental correction term in this analysis. 
A simple criterion for Condition \ref{ass:R2} to be satisfied is the following.

\begin{criterion}[A simple criterion for Condition \ref{ass:R2}] \label{rem:criterion_for_ass_R2}
    If $H^0$ has square-integrable jumps, e.g. if it is a square-integrable semimartingale and therefore has square-integrable quadratic variation, then Condition \ref{ass:R2} is satisfied. Indeed, a simple application of the Cauchy-Schwarz inequality shows that the processes 
    $$\sum_{s\le \bullet} |\Delta H^0_s \Delta X^0_s| \ind_{\{|\Delta H^0_s|\, > \, a_k\}}, \qquad k\ge 1,$$ 
    almost surely form a Cauchy sequence with respect to the uniform norm on compacts.
\end{criterion}

Let $\{a_k\} \subseteq (0,\infty)$ be a strictly decreasing null sequence to be specified at a later stage. For a càdlàg stochastic process $\tilde{H}^0$ on a probability space $(\tilde{\Omega}^0, \tilde{\mathcal{F}}^0, \tilde{\Pro}^0)$ equipped with the natural filtration of $\tilde{H}^0$, let $\tilde{\tau}^{0,k}_j$ be the stopping time indicating the $j$-th jump of $\tilde{H}^0$ of size larger than $a_k$, i.e.
\begin{equation}
    \tilde{\tau}^{0,k}_j \; := \; \inf{}\{s>\tilde{\tau}^{0,k}_{j-1} \,: \,   |\Delta \tilde{H}^0_s|> a_k \}, \quad j\ge 1,\label{eq:def_limiting_stopping_times_tightness_without_AVCI}
\end{equation}
as well as $\tilde \sigma^{0,k}_j$ the stopping time indicating the $j$-th jump of $\tilde H^0$ of size larger than $a_k$ yet smaller or equal to $a_{k-1}$, i.e.
\begin{equation}
    \tilde \sigma^{0,k}_j \; := \; \inf{}\{s>\tilde \sigma^{0,k}_{j-1} \,: \,   a_k < |\Delta H^0_s| \le a_{k-1} \}, \quad j\ge 1. \label{eq:def_limiting_stopping_times_intervals_tightness_without_AVCI}
\end{equation} 
Clearly, these stopping times almost surely exhaust all jump times of the process $\tilde{H}^0$
. Denote the set of distinct stopping times $\tilde{\sigma}^{0,k}_j$ by $\mathcal{T}(\tilde{H}^0)$, and write $\mathcal{L}(X)$ to denote the law of the random variable $X$.

\begin{theorem} \label{thm:tightness_without_AVCI}
    Let $\{X^n\}$ be a sequence of $d$-dimensional semimartingales with good decompositions \eqref{eq:Mn_An_condition} on filtered probability spaces $(\Omega^n, \mathcal{F}^n, \F^n, \Pro^n)$. Consider a sequence of adapted $d$-dimensional càdlàg processes $\{H^n\}$ on the same filtered probability space such that $(H^n,X^n) \Rightarrow (H^0,X^0)$ on $(\D_{\R^d}[0,\infty), \dM)^2$ and assume Conditions \ref{ass:R1} and \ref{ass:R2}.\\
    Then, there exists a subsequence $\{(H^{n_m},X^{n_m})\}_{m\ge 1}$, a probability space $(\tilde{\Omega}^0, \tilde{\mathcal{F}}^0,\tilde{\Pro}^0)$ as well as stochastic processes $\tilde H^0, \tilde X^0: \tilde{\Omega}^0 \to \D_{\R^d}[0,\infty)$ such that $\mathcal{L}(H^0,X^0)= \mathcal{L}(\tilde H^0, \tilde X^0)$, $\tilde X^0$ is a semimartingale in the filtration generated by the pair $(\tilde H^0, \tilde X^0)$ and 
        \begin{align} \Big(H^{n_m},\,X^{n_m}, \,\int_0^\bullet \, H^{n_m}_{s-} \; \diff X^{n_m}_s \Big) \; \Rightarrow \; \Big(  \tilde{H}^0,\,\tilde X^0, \,\int_0^\bullet \, \tilde H^0_{s-} \; \diff \tilde X^0_s \; + \hspace{-1ex} \sum_{\sigma \in \mathcal{T}(\tilde H^0)} \hspace{-1.5ex} \tilde \xi^0_{\sigma} \Delta \tilde H^0_{\sigma}\,  \Delta \tilde X^0_{\sigma} \, \ind_{[\sigma,\infty) }\Big) \label{eq:tightness_int_conv}
    \end{align}
    on $(\D_{\R^{d}}[0,\infty), \dM)^3$ as $m\to \infty$, where $0\le \tilde \xi^0_\sigma \le 1$ coordinatewise and the $\tilde \xi^0_\sigma$ are random variables on $(\tilde{\Omega}^0, \tilde{\mathcal{F}}^0,\tilde{\Pro}^0)$. In fact, $\tilde X^0$ is a semimartingale with respect to the filtration generated by the tuple $(\tilde H^0, \tilde X^0, \{\tilde \xi^0_\sigma \ind_{[\sigma , \infty)}\}_{\sigma \in \mathcal{T}(\tilde H^0)})$.
\end{theorem}

    We shall point out that the only quantities on the right-hand side of \eqref{eq:tightness_int_conv} which, in terms of their law, may differ depending on the specific subsequence $\{ (H^{n_m},X^{n_m})\}$, are the random variables $\tilde \xi^0_\sigma$. For an intuition on the nature of these sets, we refer back to Example \ref{ex:simple_example_tightness_without_AVCI}.

Figure \ref{fig:overview_tightness_results} displays a concise classification of existing tightness results for Itô integrals on the Skorokhod space, indicating how Theorem\ref{thm:tightness_without_AVCI} and the later stated Theorem \ref{cor:main_result_M1_J1_case} fit into it.

\begin{figure}[h!]
\centering 
\begin{tabular}{l c c c c} 
\hline\hline   
 &$J_1$--$J_1$ & \color{blue}{$M_1$--$J_1$} & \color{blue}{$M_1$--$M_1$}  &$S$--$S$  
\\ [0.5ex]  
\hline   
Restrictiveness of topology &high  
& \color{blue}{lower} & \color{blue}{lower}  &low \\[1ex] 
Relative compactness/topology &\checkmark/$J_1$  
& \color{blue}{\checkmark/$J_1$} & \color{blue}{\checkmark/$M_1$} &\checkmark/$S$  \\[1ex] 
Identification of limit &\checkmark  
& \color{blue}{\checkmark}& \color{blue}{\checkmark} &\ding{55}  \\[1ex] 
Continuity of usual functionals \qquad \qquad  &\checkmark  
& \color{blue}{\checkmark} & \color{blue}{\checkmark} &\ding{55} \\[1ex] 
Additional regularity required &\ding{55} &{ \color{blue}\ding{55}}  &{\footnotesize \color{blue}{$H^0 \Delta H^0 |\Delta X^0| \ge 0$}}  & \ding{55} \\ 
\hline 
\end{tabular} 
\caption{Overview of functional tightness results for stochastic integrals with respect to varying integrand-integrator convergence. The blue columns display the main results given in Section \ref{sec:main_results}.}
\label{fig:overview_tightness_results}
\end{figure}

The following proposition shows that entirely ruling out consecutive increments of the integrators before the integrands asymptotically yields again full weak convergence of the stochastic integrals, yet with a fully-fledged correction term. In terms of the limit of suitable subsequences of integrals, omitting \eqref{eq:oscillcond} therefore may well be comprehended as ending up on a random spectrum with respect to the presence of the correction term on the right-hand side of \eqref{eq:int_conv_without_AVCI}.
\begin{proposition} \label{cor:integral_conv_without_AVCI} 
    Under the assumptions of Theorem \ref{thm:tightness_without_AVCI}, suppose that asymptotically there are almost surely no consecutive large increments of the $X^n$ before the $H^n$, i.e.,
    \begin{align} \lim\limits_{\delta \downarrow 0} \, \limsup\limits_{n\to \infty} \; \Pro^n \big( \hat{w}^T_\delta(X^n,H^n_-) \, > \, a\big) \; = \; 0 \label{eq:cond_integral_conv_without_AVCI}
    \end{align}
    for all $a>0$ and $T>0$, where $\hat{w}^T_\delta$ is defined in \eqref{defi:AVCI_modulus_of_continuity}. Then, $X^0$ is a semimartingale in the natural filtration generated by the pair $(H^0,X^0)$ and it holds
    \begin{align} \Big( H^n, \,X^n , \, \int_0^\bullet \, H^n_{s-} \; \diff X^n_s\Big) \; \; \Rightarrow \; \; \Big( H^0, \,X^0,\, \int_0^\bullet \, H^0_{s-} \; \diff X^0_s \; + \; \sum_{s\le \bullet} \, \Delta H^0_s \Delta X^0_s \Big) \label{eq:int_conv_without_AVCI}
    \end{align}
    on $(\D_{\R^{d}}[0,\infty), \dM)^3$.
\end{proposition}

As pointed out in the first part of this Section, we suspect that Condition \ref{ass:R1} can be weakened without affecting the validity of Theorem \ref{thm:tightness_without_AVCI} and Proposition \ref{cor:integral_conv_without_AVCI}. A formal proof of this extension remains an open problem.

\begin{Conjecture}{\hspace{-0.5ex}} \label{conj}
    If Condition \ref{ass:R1} is replaced by the following weaker condition, then Theorem \ref{thm:tightness_without_AVCI} and Proposition \ref{cor:integral_conv_without_AVCI} still remain valid: for the limiting processes $H^0,X^0$ it holds coordinatewise that for any $s>0$,
    $$ |\Delta H^0_s |  \ind_{\{|\Delta X^0_s|\, > \, 0\}} \ind_{\{\operatorname{sgn}(\Delta H^0_s) \, \neq\,  \operatorname{sgn}(H^0_{s-})\}}\; \le \; |H^0_{s-}|  \qquad \text{$\Pro^0$-a.s.\,.}$$
\end{Conjecture}

\subsection{The $M_1$--$J_1$ case}
Provided the integrators $X^n$ converge on the more restrictive $J_1$ Skorokhod space, we are able to recover enough structure to ensure weak relative compactness, even for the $J_1$ topology, \emph{without} Condition \ref{ass:R1}. Hence, this leads to versions of Theorem \ref{thm:tightness_without_AVCI} and Proposition \ref{cor:integral_conv_without_AVCI} which do not require Condition \ref{ass:R1}. This result is stated through the following corollary.

\begin{theorem}
\label{cor:main_result_M1_J1_case}
    If $(H^n,X^n) \Rightarrow (H^0,X^0)$ on $(\D_{\R^d}[0,\infty), \dM)\times (\D_{\R^d}[0,\infty), \dJ)$, then both Theorem \ref{thm:tightness_without_AVCI} and Proposition \ref{cor:integral_conv_without_AVCI} remain valid even \underline{without} Condition \ref{ass:R1}, and the convergence \eqref{eq:tightness_int_conv}, or \eqref{eq:int_conv_without_AVCI} respectively, hold on the stronger space $(\D_{\R^d}[0,\infty), \dM)\times (\D_{\R^{2d}}[0,\infty), \dJ)$.
\end{theorem}

\section{Proofs} \label{sec:proofs}
Define
\begin{align*}
    \operatorname{Disc}_{\Pro^0}(H^0,X^0) \; := \; \Big\{s>0 \; : \; \Pro^0\big(|\Delta H^0_s|\vee |\Delta X^0_s| > 0 \big) \, > \, 0 \Big\}  
\end{align*}
with the usual notation $a\vee b:= \operatorname{max}\{a,b\}$, and recall that this set is at most countable (see e.g. \cite[Lem.~VI.3.12]{shiryaev}).
Towards a proof of Theorem \ref{thm:tightness_without_AVCI} and Proposition \ref{cor:integral_conv_without_AVCI}, we construct the subsequent quantities. Let $\{\Theta_\ell\}$ be a sequence of nested partitions such that 
\begin{align*} 
    \Theta_\ell \; = \; \{\nu^\ell_i \, : \, i\ge 1, \, 0= \nu^\ell_1&<\nu^\ell_2<..., \,\text{ with } \nu^\ell_i \to \infty \text{ as } i\to \infty\} \; \subseteq \;  [0, \infty) \setminus \operatorname{Disc}_{\Pro^0}(H^0,X^0) 
\end{align*}
with vanishing mesh size, i.e., $|\Theta_\ell|:=\operatorname{sup}_{i\ge 2}(\nu^\ell_{i}-\nu^\ell_{i-1}) \to 0$ as $\ell \to \infty$, $|\Theta_\ell| \notin \{|x-y|: x,y \in \operatorname{Disc}_{\Pro^0}(H^0,X^0)\}$, and $|\Theta_\ell|^{\downarrow}:=\operatorname{inf}_{i\ge 2}(\nu^\ell_{i}-\nu^\ell_{i-1})>0$ for all $\ell \ge 1$. Further, define
$$ 
\lceil x \rceil_{\Theta_\ell}\; := \; \begin{cases}
\operatorname{min}\{\nu \in \Theta_\ell \, : \, \nu \ge x\} \quad&, \text{if } x \notin \Theta_\ell\\
\operatorname{min}\{\nu \in \Theta_\ell: \nu \ge x+|\Theta_\ell|^{\downarrow}/2\}\quad&, \text{if }  x \in \Theta_\ell
\end{cases}
$$
and analogously for $\lfloor x \rfloor_{\Theta_\ell}$, where we simply reverse the inequalities, replace the minimum by a maximum, and $|\Theta_\ell|^{\downarrow}/2$ by $-|\Theta_\ell|^{\downarrow}/2$. The distinction made for the second case $x\in \Theta_\ell$ serves purely technical purposes relevant to the proof of Lemma \ref{lem:conv_relevant_quant_remainder_integrals}.
Let $\{a_k\}\subseteq (0,\infty)$ be a strictly decreasing null sequence, which---for purely technical reasons---will later be chosen in a convenient non-constructive way. The specific choice of the sequence is only relevant to ensure the convergence in Lemma \ref{lem:conv_relevant_quant_remainder_integrals}, and can be disregarded throughout all other instances of appearance of the sequence $\{a_k\}$ in this paper.

To given $k,\ell\ge 1$, $n\ge 0$ set $\rho^{n,k,\ell}_0\equiv 0$ and define iteratively  
\begin{align*}
    \tau^{n,k,\ell}_j  &:=  \operatorname{inf}\biggl\{t>\rho^{n,k,\ell}_{j-1}  : \operatorname{sup} \big\{ |H^{n,(i)}_t-H^{n,(i)}_s|  :  \rho^{n,k,\ell}_{j-1} \vee (t- |\Theta_\ell|) \le s \le t, 1\le i\le d \big\} > a_{k} \biggr\} \\[1ex]
    \rho^{n,k,\ell}_j  &:=  \lceil \tau^{n,k,\ell}_j\rceil_{\Theta_\ell} 
\end{align*}
for all $j\ge 1$, where $H^{n,(i)}$ denotes the $i$-th coordinate of $H^n$. Then, we can set  
\begin{align}
    \tilde{H}^{n,k,\ell}_t \; := \;   \begin{cases} H^n_t-H^n_{\tau^{n,k,\ell}_j-} \qquad &,\text{if } \tau^{n,k,\ell}_j \le t < \rho^{n,k,\ell}_j \text{ for some } j\ge 1,\\
    0 \qquad &,\text{otherwise}. \end{cases} \label{eq:definition_correction_processes}
\end{align}
Note that for fixed $k,\ell\ge 1$, $n\ge 0$, there are almost surely only finitely many $\tau_j$ on compact time intervals. The process $\tilde{H}^{n,k,\ell}$ simply describes the trajectory of $H^n$ between each $\tau^{n,k,\ell}_j$ and $\rho^{n,k,\ell}_j$, if we start it from "last" observed value of $H^n$ before $\tau^{n,k,\ell}_j$, while it is set to zero outside of these intervals. To illustrate the procedure more concretely, consider $j=1$: $\tau^{n,k,\ell}_1$ describes the time at which the process, for the first time, has had an increment larger than $a_k$. On the interval $[0,\tau^{n,k,\ell}_1)$ the auxiliary process $\tilde{H}^{n,k,\ell}$ is equal to zero. Then, it replicates the increments of $H^n$ with respect to the left-limit of the latter at $\tau^{n,k,\ell}_1$ until we reach the next temporal grid point from $\Theta_\ell$, namely $\rho^{n,k,\ell}_1$. From then on $\tilde{H}^{n,k,\ell}$ equals again zero until time $\tau^{n,k,\ell}_2$, where this replication starts anew. 

We can now proceed to the proof of the main results, and we will do so in several auxiliary steps (Sections \ref{sec:proofs_corrected_integrands} \& \ref{sec:proof_dealing_with_the_remainder_integrals}). The final proofs of the main results will then be presented in Sections \ref{sec:proof_main_result}, \ref{sec:proof_proposition_anti_AVCI} and \ref{sec:proof_corollary_J1_integrators_main_result}.

\subsection{Relative compactness for the corrected integrands} \label{sec:proofs_corrected_integrands}
\begin{lemma} \label{lem:tightness_of_corrected_integrands_1}
    For fixed $k,\ell \ge 1$, the sequence $\{H^n \, - \, \tilde{H}^{n,k,\ell}\}_{n\ge 1}$ is tight on $(\D_{\R^{d}}[0,\infty), \dM)$.
\end{lemma}
\begin{proof}
    Recall the notation $|Z|^*_t:= \operatorname{sup}_{0\le s\le t} |Z_s|$. According to \cite[Thm.~12.12.3]{whitt}, the sequence $\{H^n \, - \, \tilde{H}^{n,k,\ell}\}_{n\ge 1}$ is $M_1$ tight if and only if, for every $T>0$ in a dense subset of $[0,\infty)$,
    \begin{enumerate}[(1)]
        \item \label{it1:tightness_of_corrected_integrands_1} the sequence is stochastically bounded on $[0,T]$, that is
        $$ \lim_{R\to \infty} \, \limsup\limits_{n\to \infty} \, \Pro^n\big(|H^n \, - \, \tilde{H}^{n,k,\ell}|^*_T>R\big) \, = \, 0,$$
        \item \label{it2:tightness_of_corrected_integrands_1}the $M_1$ modulus of continuity vanishes asymptotically in probability on increasingly small intervals, that is for all $\eta>0$ it holds that
        $$ \lim_{\delta \downarrow 0} \, \limsup\limits_{n\to \infty} \, \Pro^n\big(w'(H^n \, - \, \tilde{H}^{n,k,\ell}, \delta) > \eta \big) \, = \, 0,$$
        where $w'(x,\delta):= \operatorname{sup}\{ \, \norm{x_t-[x_s,x_r]} :  0\vee (t-\delta) \le s\le t\le r \le (t+\delta)\wedge T\}$ with $\norm{x_t-[x_s,x_r]}:= \operatorname{inf}\{|x_t-y | : 0\le \lambda\le 1, \, y=\lambda x_s + (1-\lambda)x_r \}$.
        \item \label{it3:tightness_of_corrected_integrands_1} there is local uniform convergence at $0$ and $T$, which means for all $\eta>0$ it holds that
        $$ \lim_{\delta \downarrow 0} \, \limsup\limits_{n\to \infty} \, \Pro^n\big(u(H^n \, - \, \tilde{H}^{n,k,\ell},0,\delta) \vee u(H^n \, - \, \tilde{H}^{n,k,\ell},T,\delta)  > \eta \big) \, = \, 0, $$
        where $u(x,t,\delta):= \operatorname{sup}\{ |x_r-x_s| : 0\vee (t-\delta) \le r,s \le (t+\delta)\wedge T\}$.
    \end{enumerate}
    Since $H^n \Rightarrow H^0$ on $(\D_{\R^{d}}[0,\infty), \dM)$, an application of Skorokhod's representation theorem together with \cite[Thm.~12.5.1(v)]{whitt} gives us that for every $t \in [0,\infty)\setminus \operatorname{Disc}_{\Pro^0}(H^0)$ there is local uniform tightness at $t$ of the type
    \begin{align}
        \lim_{\delta \downarrow 0} \, \limsup\limits_{n\to \infty} \, \Pro^n\big(u(H^n,t, \delta) > \eta \big) \, = \, 0 \label{eq:local_uniform_tightness_at_continuity_points_corrected_integrands_1}
    \end{align}
    for all $\eta>0$. Recall that $\operatorname{Disc}_{\Pro^0}(H^0)$ is a countable subset of $[0,\infty)$ and let $T>0$, $T \in [0,\infty) \setminus \operatorname{Disc}_{\Pro^0}(H^0)$. Clearly, due to the tightness of the sequence $\{H^n\}$ on the $M_1$ Skorokhod space, it satisfies the conditions \eqref{it1:tightness_of_corrected_integrands_1}--\eqref{it3:tightness_of_corrected_integrands_1} with $H^n \, - \, \tilde{H}^{n,k,\ell}$ replaced by $H^n$. Then, \eqref{it1:tightness_of_corrected_integrands_1} and \eqref{it3:tightness_of_corrected_integrands_1} follow immediately---by definition of $\tilde{H}^{n,k,\ell}$---from $|H^n \, - \, \tilde{H}^{n,k,\ell}|^*_T \le |H^n|^*_T$ and $u(H^n \, - \, \tilde{H}^{n,k,\ell},0,\delta) \vee u(H^n \, - \, \tilde{H}^{n,k,\ell},T,\delta)\le u(H^n ,0,\delta)\vee u(H^n ,T,\delta)$. Thus, it only remains to show \eqref{it2:tightness_of_corrected_integrands_1}. Towards this goal, let $\varepsilon,\eta>0$ and choose $\delta>0$, $N\ge 1$ such that $\Pro^n(A_{n,\delta}) > 1-\varepsilon$ for all $n\ge N$, where 
    $$ A_{n,\delta} \; := \; \Big\{ w'(H^n, \delta)\, \vee  \, 2\operatorname{max}_{\nu \in \Theta_\ell \cap [0,T]} u(H^n,\nu, \delta)  \;\le \; \frac{\eta}{2}\Big\}. $$
    Indeed, this choice is possible due to \eqref{eq:local_uniform_tightness_at_continuity_points_corrected_integrands_1} and the fact that $\{H^n\}$ is $M_1$ tight and therefore satisfies \eqref{it2:tightness_of_corrected_integrands_1} with $H^n \, - \, \tilde{H}^{n,k,\ell}$ replaced by $H^n$ (note that $\Theta_\ell \cap [0,T]$ is finite). Now, let $r,s,t>0$ be such that $0\vee (t-\delta) \le s\le t\le r \le (t+\delta)\wedge T$ and fix $\omega \in \Omega^n$. First, consider the case $\rho^{n,k,\ell}_{j-1}(\omega) \le s < \tau^{n,k,\ell}_j(\omega) \le t < \rho^{n,k,\ell}_j(\omega)\le r < \tau^{n,k,\ell}_{j+1}(\omega)$ for some $j\ge 1$, for which
    \begin{align*}
        &\operatorname{inf}_{\lambda \in [0,1]} |H^n_t-\tilde{H}^{n,k,\ell}_t - \lambda (H^n_s - \tilde{H}^{n,k,\ell}_s) -(1-\lambda)(H^n_r -\tilde{H}^{n,k,\ell}_r)| \\
        \le \; &\operatorname{inf}_{\lambda \in [0,1]} |H^n_{\tau^{n,k,\ell}_j-} - \lambda H^n_s -(1-\lambda)H^n_r | \; \le \; w'(H^n,\delta),
    \end{align*}
    where we have omitted the dependence of all quantities on $\omega$. If $\tau^{n,k,\ell}_j(\omega) \le s < \rho^{n,k,\ell}_j(\omega) \le t < r < \tau^{n,k,\ell}_{j+1}(\omega)$ for some $j\ge 1$, then 
    \begin{align}
        &\operatorname{inf}_{\lambda \in [0,1]} \, |H^n_t-\tilde{H}^{n,k,\ell}_t - \lambda (H^n_s - \tilde{H}^{n,k,\ell}_s) -(1-\lambda)(H^n_r -\tilde{H}^{n,k,\ell}_r)| \notag \\
        &\operatorname{inf}_{\lambda \in [0,1]} \, |H^n_t-\tilde{H}^{n,k,\ell}_t - (H^n_s - \tilde{H}^{n,k,\ell}_s) -(1-\lambda)(H^n_r -\tilde{H}^{n,k,\ell}_r)+ (1-\lambda)(H^n_s - \tilde{H}^{n,k,\ell}_s)| \notag \\
        \le \; & \operatorname{inf}_{\lambda \in [0,1]} \, \big( |H^n_t - H^n_s| \; + \; |H^n_r - H^n_s| \; + \; |H^n_{\tau^{n,k,\ell}_j-}-(1-\lambda)H^n_{\tau^{n,k,\ell}_j-}|\big) \notag \\
        \le \; &2 \operatorname{max}_{\nu \in \Theta_\ell \cap [0,T]} u(H^n,\nu, \delta) \label{eq:different_bound_than_second_tightness_result_corrected_integrands}
    \end{align}
    since $\rho^{n,k,\ell}_j \in \Theta_\ell$ and $0\vee(\rho^{n,k,\ell}_j-\delta)\le r,s \le (\rho^{n,k,\ell}_j+\delta)\wedge T$. For the sake of brevity, we do not provide details on the remaining cases and we shall simply note that such bound can be derived in similar ways for them, so to obtain
    $$ w'(H^n-\tilde{H}^{n,k,\ell},\delta) \; \le \; w'(H^n,\delta) \, + \, 2\operatorname{max}_{\nu \in \Theta_\ell \cap [0,T]} u(H^n,\nu, \delta)$$
    almost surely. Hence, $1-\varepsilon < \Pro^n(A_{n,\delta}) \le \Pro^n(w'(H^n-\tilde{H}^{n,k,\ell},\delta)\le\eta)$ for all $n\ge N$. Since 
    $\varepsilon >0$ was arbitrary, we deduce \eqref{it2:tightness_of_corrected_integrands_1}.
\end{proof}

By Prokhorov's Theorem, from Lemma \ref{lem:tightness_of_corrected_integrands_1} we deduce that, for any $k,\ell \ge 1$ and subsequence of $\{H^{n}-\tilde{H}^{n,k,\ell}\}_{n\ge 1}$, there exists a further subsequence $\{H^{n_{m}(k,\ell)}-\tilde{H}^{n_m(k,\ell),k,\ell}\}_{m\ge 1}$ and a càdlàg stochastic process $Y^{k, \ell}$ such that 
\begin{align}
    H^{n_m(k,\ell)}-\tilde{H}^{n_m(k,\ell),k,\ell} \; \Rightarrow \; Y^{k, \ell} \label{eq:convergence_tightness_of_corrected_integrands_1}
\end{align}
on $(\D_{\R^{d}}[0,\infty), \dM)$ as $m\to \infty$. We shall stress that, even if not explicitly marked, the limt $Y^{k,\ell}$ may---a priori---depend on the specific subsequence $\{n_m(k,\ell)\}_{m\ge 1}$. Moreover, by definition of the processes \eqref{eq:definition_correction_processes}, $H^{n}_\nu-\tilde{H}^{n,k,\ell}_\nu=H^n_\nu$ for all $\nu \in \Theta_\ell$, and therefore the finite-dimensional distributions of $H^{n_m}-\tilde{H}^{n_m(k,\ell),k,\ell}$ converge to those of $H^0$ on $\Theta_\ell \subseteq [0,\infty) \setminus \operatorname{Disc}_{\Pro^0}(H^0,X^0)$. 


\begin{lemma} \label{lem:tightness_of_corrected_integrands_2}
    For fixed $k \ge 1$ and a sequence $\{\ell_r\}$ with $\ell_r\to \infty$ as $r\to \infty$, the sequence of respective limits $\{Y^{k,\ell_r}\}_{r\ge 1}$ from \eqref{eq:convergence_tightness_of_corrected_integrands_1} is tight on $(\D_{\R^{d}}[0,\infty), \dM)$.
\end{lemma}
\begin{proof}
Let $\varepsilon,\eta>0$. By \eqref{eq:convergence_tightness_of_corrected_integrands_1} and since weak convergence is metrisable, there exists $b:=b(r)$ increasing in $r$ such that $d_w(H^{n_b(k,\ell_r)}-\tilde H^{n_b(k,\ell_r),k,\ell_r}, Y^{k,\ell_r}) \le 1/r$, where $d_w$ denotes any metric generating the topology of weak convergence of Borel probability measures on $(\D_{\R^{d}}[0,\infty), \dM)$. Thus, in order to show $M_1$ tightness of the sequence $\{Y^{k,\ell_r}\}_{r\ge 1}$ it suffices to show the tightness of the sequence $\{H^{n_b(k,\ell_r)}-\tilde H^{n_b(k,\ell_r),k,\ell_r}\}_{r\ge 1}$. Without loss of generality, we may assume that $b$ is such that there exists $\delta>0$ with 
$$ \Pro^{n_b(k,\ell_r)} \Big( w'\big(H^{n_b(k,\ell_r)}, \delta\big)\;\le \; \eta \Big) \; > \; 1- \varepsilon$$
for all $r\ge 1$ (as a result of the $M_1$ tightness of the $H^n$; see the partial condition \eqref{it2:tightness_of_corrected_integrands_1} in the proof of Lemma \ref{lem:tightness_of_corrected_integrands_1}). Let us write $n_b$ instead of $n_b(k,\ell_r)$ for brevity. Now, proceed in analogy to the proof of Lemma \ref{lem:tightness_of_corrected_integrands_1}, noting that instead of the bound in \eqref{eq:different_bound_than_second_tightness_result_corrected_integrands} we achieve
    \begin{align*}
        &\operatorname{inf}_{\lambda \in [0,1]} \, \big|H^{n_b}_t-\tilde{H}^{n_b,k,\ell_r}_t - \lambda (H^{n_b}_s - \tilde{H}^{n_b,k,\ell_r}_s) -(1-\lambda)(H^{n_b}_p -\tilde{H}^{n_b,k,\ell_r}_p)\big|  \\
        \le \; & \operatorname{inf}_{\lambda \in [0,1]} \, |H^{n_b}_t - \lambda H^{n_b}_{\tau^{n_b,k,\ell_r}_j-} -(1-\lambda)H^{n_b}_p|  \\
        \le \; &w'(H^{n_b},\delta/2+ 2|\Theta_{\ell_r}|)
    \end{align*}
    whenever $\tau^{n_b,k,\ell_r}_j \le s < \rho^{n_b,k,\ell_r}_j \le t < p < \tau^{n_b,k,\ell_r}_{j+1}$ for some $j\ge 1$ and $p-s\le \delta/2$, since $\rho^{n_b,k,\ell_r}_j-\tau^{n_b,k,\ell_r}_j\le |\Theta_{\ell_r}|$ and therefore $p-\tau^{n_b,k,\ell_r}_j<\delta/2+2|\Theta_{\ell_r}|$. Recall that $|\Theta_{\ell_r}|$ denotes the mesh size of the partition $\Theta_{\ell_r}$. In fact, all other cases for $s,t,p$ with $p-s\le \delta/2$ yield the same bound (again, the procedure is similar and we omit the details for the sake of brevity of this account). Thus, we obtain
    $$ w'\big(H^{n_b(k,\ell_r)}-\tilde H^{n_b(k,\ell_r),k,\ell_r},\delta/2\big) \; \le \; w'\big(H^{n_b(k,\ell_r)}, \delta\big)$$
    for $r$ large enough, due to $|\Theta_{\ell_r}|\to 0$ as $r\to \infty$.
\end{proof}

Again, by Prokhorov's Theorem and Lemma \ref{lem:tightness_of_corrected_integrands_2}, to every subsequence of $\{\ell_r\}$ there exists a further subsequence---which, whenever there is no risk of ambiguity, we will denote as the original sequence---converging weakly in $(\D_{\R^{d}}[0,\infty), \dM)$ to some limit $Y^{k}$, that is
\begin{align}
    Y^{k,\ell_r} \; \Rightarrow \; Y^{k} \label{eq:convergence_2_tightness_of_corrected_integrands_1}
\end{align}
as $r\to \infty$. Again, while not stated explicitly through the notation, the limit $Y^k$ may---a priori---depend on the choice of the particular subsequence. We will now show that the laws of $Y^{k}$ and $H^0$ coincide, and therefore the convergence in 
\eqref{eq:convergence_2_tightness_of_corrected_integrands_1}---a posteriori---turns out not to depend on the actual subsequence of $\{\ell_r\}$, yielding in fact
\begin{align}
    Y^{k,\ell_r} \; \Rightarrow \; H^0 \label{eq:convergence_3_tightness_of_corrected_integrands_1}
\end{align}
as $r \to \infty$ for any original sequence $\{\ell_r\}$. Recall that the Borel $\sigma$-algebra on the Skorokhod space is generated by the marginal time projections onto a dense subset of times, and the set $\bigcup_{\ell \ge 1} \Theta_\ell$ is dense in $[0,\infty)$.

\begin{lemma}
For any $k\ge 1$ and limit $Y^{k}$ from \eqref{eq:convergence_2_tightness_of_corrected_integrands_1} the finite-dimensional distributions of $Y^{k}$ and $H^0$ coincide on $\Theta_\ell$, that is 
\begin{align}
    \mathcal{L}\Big(Y^{k}_{\nu_1},...,Y^{k}_{\nu_r}\Big) \; = \; \mathcal{L}\big(H^0_{\nu_1}, ..., H^0_{\nu_r}\big)
\end{align}
for all $p\ge 1$ and $\nu_1,...,\nu_p \in \bigcup_{\ell\ge 1}\Theta_\ell$.
\end{lemma}
\begin{proof}
    Let $k,p\ge 1$ and $\nu_1,...,\nu_p \in \Theta_{\ell}$, for some $\ell\ge 1$. In order to lighten notation, we will simply write $Y$ to denote $Y^{k}$. By \eqref{eq:convergence_tightness_of_corrected_integrands_1}, \eqref{eq:convergence_2_tightness_of_corrected_integrands_1}, and since weak convergence on the $M_1$ space is metrisable, again there exist $b:=b(r)$ such that $n_b(k,\ell_r) \to \infty$ as $r\to \infty$, and 
     $$ \hat{H}^{r} \; := \; H^{n_b(k,\ell_r)}-\tilde{H}^{n_b(k,\ell_r),k,\ell_r} \; \Rightarrow \; Y$$
     on $(\D_{\R^{d}}[0,\infty), \dM)$ as $r\to \infty$. Moreover, in order to further ease notation, we will suppress the dependence of $n_b(k,\ell_r)$ on $k,r$ and simply write $n_b$ in the sequel. Recall that $D:= [0,\infty) \setminus (\operatorname{Disc}_{\Pro}(Y) \,\cup \,\operatorname{Disc}_{\Pro^0}(H^0))$ is cocountable and choose a sequence of times $\{(t^q_1,...,t^q_p)\}_{q\ge 1}$ such that $(t^q_1,...,t^q_p) \in D^p$, $t^{q}_j \downarrow \nu_j$ as $q\to \infty$, and $t^q_j\neq \nu_j$ for every $j=1,...,p$. Let $f:\R^{pd} \to \R$ be Lipschitz and bounded, and denote $f_\nu=f\circ \pi_{\nu_1,...,\nu_p}$, $f_{t^q}=f\circ \pi_{t^q_1,...,t^q_p}$, where $\pi_{\alpha_1,...,\alpha_p}(x):=(x_{\alpha_1},...,x_{\alpha_p})$ is the coordinate projection at times $(\alpha_1,...,\alpha_p)$. Then, we obtain the bound
    \begin{align}
        \big| \E[f_\nu(Y)] -  \E^{0}[f_\nu(H^0)] \big| \; \le \; &\big| \E[f_\nu(Y)] -  \E[f_{t^q}(Y)] \big|  +  \big| \E[f_{t^q}(Y)] -  \E^{n_b}[f_{t^q}(\hat{H}^{r})] \big| \notag \\ &+ \big| \E^{n_b}[f_{t^q}(\hat{H}^{r})] -  \E^{n_b}[f_{\nu}(\hat{H}^{r})] \big|  + \big| \E^{n_b}[f_{\nu}(\hat{H}^{r})] -  \E^{0}[f_\nu(H^0)] \big| \label{eq:auxiliary_equation_equal_finite_dimensional_distributions_tightness_of_corrected_integrands_1}
    \end{align}
    for all $r,q\ge 1$. Since $f_{\nu}(\hat{H}^{r})=f_\nu(H^{n_b})$ by definition of the processes $\hat{H}^r$, and $\nu_j \in [0,\infty)\setminus \operatorname{Disc}_{\Pro^0}(H^0)$, $j=1,...,p$, the generalised continuous mapping theorem \cite[Thm.~2.7]{billingsley} yields that the last term on the right side of \eqref{eq:auxiliary_equation_equal_finite_dimensional_distributions_tightness_of_corrected_integrands_1} converges to zero as $r\to \infty$ (and therefore $n_b \to \infty$), as a consequence of $H^{n_b} \Rightarrow H^0$ on $(\D_{\R^d}[0,\infty), \dM)$ and \cite[Thm.~12.5.1(v)]{whitt}. By the same reasoning, the second term on the right-hand side of \eqref{eq:auxiliary_equation_equal_finite_dimensional_distributions_tightness_of_corrected_integrands_1} converges to zero as $r\to \infty$, for each $q\ge 1$. The first term simply follows by dominated convergence and the right-continuity of $Y$. Finally, the third term can be estimated by 
    \begin{align*}
        \big| \E^{n_b}[f_{t^q}(\hat{H}^{r})] -  \E^{n_b}[f_{\nu}(\hat{H}^{r})] \big| \; &\le \; \norm{f}_{\operatorname{Lip}} \E^{n_b}\big[|H^{n_b}_\nu - \hat{H}^{r}_{t^q}| \wedge 1\big] 
    \end{align*}
    where $\norm{f}_{\operatorname{Lip}}$ is the maximum of the Lipschitz constant and bound of $f$. By definition of the processes $\hat{H}^r$ (and more precisely \eqref{eq:definition_correction_processes}), note that $|\hat H^{n_b}_{\nu_j} - \hat{H}^{r}_{t^q_j}|\le u(H^{n_b}, \nu_j, t^q_j-\nu_j)$, where $u$ is defined as in \eqref{it3:tightness_of_corrected_integrands_1} in the proof of Lemma \ref{lem:tightness_of_corrected_integrands_1}. Hence, for every $\varepsilon>0$,
    \begin{align*}
        \big| \E^{n_b}[f_{t^q}(\hat{H}^{r})] -  \E^{n_b}[f_{\nu}(\hat{H}^{r})] \big| \; &\le \; \norm{f}_{\operatorname{Lip}} \Big(\varepsilon \; + \; \sum_{j=1}^{p} \Pro^{n_b} \big(u(H^{n_b},\nu_j, |t^q-\nu|) > \varepsilon \big)\Big),
    \end{align*}
    and, by \eqref{eq:local_uniform_tightness_at_continuity_points_corrected_integrands_1} we obtain 
    $$ \lim_{q\to \infty} \, \limsup_{r\to \infty} \, \big| \E^{n_b}[f_{t^q}(\hat{H}^{r})] -  \E^{n_b}[f_{\nu}(\hat{H}^{r})] \big| \; = \; 0$$
    as $\varepsilon>0$ was arbitrary. Thus, we obtain that the left-hand side of \eqref{eq:auxiliary_equation_equal_finite_dimensional_distributions_tightness_of_corrected_integrands_1} equals zero, which yields the claim.
\end{proof}

We can now deduce the following weak limit result.

\begin{proposition} \label{prop:convergence_of_corrected_integrands}
    For all subsequences of the natural numbers $\{k_p\}$, $\{\ell_r\}$ and $\{n_m\}$, there exist further subsequences $\{\ell_{r(p)}\}$, $\{n_{m(p)}\}$ of $\{\ell_r\}$ and $\{n_m\}$ respectively such that 
    \begin{align}
    H^{n_{m(p)}} \, - \, \tilde{H}^{n_{m(p)},k_p,\ell_{r(p)}} \; \Rightarrow \; H^0 \label{eq:convergence_of_corrected_integrands}
\end{align}
on $(\D_{\R^{d}}[0,\infty), \dM)$ as $p\to \infty$.
\end{proposition}
\begin{proof}
    Since the topology of weak convergence on $(\D_{\R^{d}}[0,\infty), \dM)$ is metrisable, the result follows directly from  \eqref{eq:convergence_3_tightness_of_corrected_integrands_1} and \eqref{eq:convergence_tightness_of_corrected_integrands_1} by making use of a diagonal sequence argument. 
\end{proof}

\begin{lemma}[An \eqref{eq:oscillcond}-type result for the approximations] \label{lem:AVCI_replacement_for_corrected_integrands_without_AVCI}
For all $\gamma>0$ and $T>0$ it holds that
\begin{align*}
    \lim\limits_{k\to \infty} \;  \sup{\ell \ge 1}\;\limsup\limits_{\delta \downarrow 0} \; \limsup\limits_{n\to \infty} \; \; \Pro^n \Big( \hat{w}^T_\delta(H^n-\tilde{H}^{n,k,\ell}, X^n) \, > \, \gamma \Big) \; = \; 0
\end{align*}
where $\hat{w}^T_\delta$ is defined in \eqref{defi:AVCI_modulus_of_continuity}.
\end{lemma}
\begin{proof}
Denote $\hat{H}^{n,k,\ell}:= H^n -\tilde{H}^{n,k,\ell}$. Let $\gamma, \varepsilon>0$, $T>0$ and choose $k\ge 1$ such that $a_k<\gamma$. Now, fix $\ell \ge 1$ and choose $\delta:=\delta(\gamma,\ell)>0$ small enough that  
$$ \limsup\limits_{n\to \infty} \, \Pro^n\big(\operatorname{max}_{\nu \in \Theta_\ell \cap [0,T]} u(X^n,\nu, 2\delta) > \gamma \big) \, < \, \varepsilon, $$
where $u$ is defined as in \eqref{it3:tightness_of_corrected_integrands_1} in the proof of Lemma \ref{lem:tightness_of_corrected_integrands_1}. The existence of such $\delta$ can be deduced as in \eqref{eq:local_uniform_tightness_at_continuity_points_corrected_integrands_1}, due to $X^n \Rightarrow X^0$ on $(\D_{\R^d}[0,\infty), \dM)$ and $\Theta_\ell\cap [0,T] \subseteq [0,\infty) \setminus \operatorname{Disc}_{\Pro^0}(X^0)$ being finite. Consider $\omega \in \{\hat{w}^T_\delta(\hat{H}^{n,k,\ell}, X^n) > \gamma\}$. Then, there exist random times $0\le s(\omega) \le t(\omega) \le p(\omega) \le T$, satisfying $p(\omega)-t(\omega) \le \delta$, $t(\omega)-s(\omega) \le \delta$, and 
\begin{align} \big|\hat{H}^{n,k,\ell}_{s(\omega)} (\omega)-\hat{H}^{n,k,\ell}_{t(\omega)}(\omega)\big| \, \wedge \, \big|X^n_{t(\omega)}(\omega) - X^n_{p(\omega)}(\omega)\big| \; > \; \gamma \; > \; a_k. \label{eq:some_eq_in_the_proof_of_AVCI_type_result}
\end{align}
Note that there must be $j\ge 1$ so that $s(\omega)< \rho^{n,k,\ell}_j(\omega) \le t(\omega)$. Indeed, otherwise, we have $\rho^{n,k,\ell}_j(\omega) \le s(\omega)< t(\omega) < \rho^{n,k,\ell}_{j+1}(\omega)$ for some $j\ge 1$. By definition of the stopping times $\rho^{n,k,\ell}_j$ before \eqref{eq:definition_correction_processes}, \eqref{eq:some_eq_in_the_proof_of_AVCI_type_result}, and the definition of the processes $\hat{H}^{n,k,\ell}$ in \eqref{eq:definition_correction_processes}, it follows $\rho^{n,k,\ell}_j(\omega) \le s(\omega)< \tau_{j+1}^{n,k,\ell}(\omega)\le t(\omega)< \rho^{n,k,\ell}_{j+1}(\omega)$. However, this implies 
$$ |\hat{H}^{n,k,\ell}_{s(\omega)} (\omega)-\hat{H}^{n,k,\ell}_{t(\omega)}(\omega)|\; = \; |H^{n}_{s(\omega)} (\omega)-H^{n}_{\tau^{n,k,\ell}(\omega)-}(\omega)| \; < \; a_k,$$
which yields a contradiction.
Thus, there exists $j\ge 1$ such that $s(\omega)< \rho^{n,k,\ell}_j(\omega) \le t(\omega) < p(\omega)$. Since $\rho^{n,k,\ell}_j(\omega) \in \Theta_\ell$ and $p(\omega)-s(\omega)<2\delta$, 
$$ \operatorname{min}_{\nu \in \Theta_\ell\cap [0,T]} \big( |p(\omega)- \nu| \vee |t(\omega)-\nu|\big) \; < \; 2\delta,$$
and therefore $\omega \in \{\operatorname{max}_{\nu \in \Theta_\ell \cap [0,T]} u(X^n,\nu, 2\delta) > \gamma\}$, since we recall $|X^n_{t(\omega)}(\omega)-X^n_{p(\omega)}(\omega)|>\gamma$. Finally, we deduce 
$$ \limsup\limits_{n\to \infty} \, \Pro^n \Big( \hat{w}^T_\delta(H^n-\tilde{H}^{n,k,\ell}, X^n)  >  \gamma \Big) \,\le \, \limsup\limits_{n\to \infty} \, \Pro^n\big(\operatorname{max}_{\nu \in \Theta_\ell \cap [0,T]} u(X^n,\nu, 2\delta) > \gamma \big)  < \varepsilon. $$
As $\varepsilon,\gamma >0$ were arbitrary, this yields the result.
\end{proof}

Finally, the prior considerations allow us to establish the following weak convergence results with respect to the corrected integrands. 
\begin{proposition} \label{cor:convergence_without_AVCI_for_corrected_integrands}
Let $\{k_p\}$, $\{n_m\}$ and $\{\ell_r\}$ be subsequences of the natural numbers. Then, there exist further subsequences $\{n_{m(p)}\}$, $\{\ell_{r(p)}\}$ of $\{n_m\}$ and $\{\ell_r\}$ respectively such that 
    \begin{align}
		\Big( H^{n_{m(p)}},\,X^{n_{m(p)}}, \, \int_0^\bullet \, H^{n_{m(p)}}_{s-} \, &- \, \tilde{H}^{n_{m(p)},k_p,\ell_{r(p)}}_{s-} \, \diff X_s^{n_{m(p)}} \Big)\; \Rightarrow \; \Big(  H^0,\,X^0, \, \int_0^\bullet \, H^0_{s-} \, \diff X^0_s \Big)\label{eq:convergence_without_AVCI_for_corrected_integrands}
	\end{align}
    on $(\D_{\R^{d}}[0,\infty), \, \dM)^3$ as $p\to \infty$.
\end{proposition}
\begin{proof}
By Proposition \ref{prop:convergence_of_corrected_integrands} there exist subsequences $\{n_{m(p)}\}$, $\{\ell_{r(p)}\}$ that ensure 
\begin{align*}
    H^{n_{m(p)}} \, - \, \tilde{H}^{n_{m(p)},k_p,\ell_{r(p)}} \; \Rightarrow \; H^0 
\end{align*}
on $(\D_{\R^{d}}[0,\infty), \dM)$ as $p\to \infty$. Without loss of generality, assume that 
$$\big(H^{n_{m(p)}},\,X^{n_{m(p)}}, \,H^{n_{m(p)}} \, - \, \tilde{H}^{n_{m(p)},k_p,\ell_p}\big) \; \Rightarrow \;  \big(H^0,X^0,H^0\big)$$ on $(\D_{\R^{d}}[0,\infty), \, \dM)^3$ as $p\to \infty$, and by Lemma \ref{lem:AVCI_replacement_for_corrected_integrands_without_AVCI} the sequence $\{(X^{n_{m(p)}}, H^{n_{m(p)}} \, - \, \tilde{H}^{n_{m(p)},k_p,\ell_p})\}$ has \eqref{eq:oscillcond}. Then, the result follows immediately from \cite[Thm.~3.6]{andreasfabrice_theorypaper}.
\end{proof}

\subsection{Dealing with the remainder integrals} \label{sec:proof_dealing_with_the_remainder_integrals}
Let $\eta,\gamma>0$, $T>0$ and $X^n=M^n+A^n$ the semimartingale decomposition of the $X^n$ satisfying \eqref{eq:Mn_An_condition}. In this section, we are going to investigate the limiting behaviour of the remainder integrals
\begin{align} \int_0^t \tilde{H}^{n,k,\ell}_{s-} \; \diff X^n_s \; = \; \sum_{j=1}^\infty \int_{t \wedge \tau^{n,k,\ell}_j }^{ t \wedge \rho^{n,k,\ell}_j} \tilde{H}^{n,k,\ell}_{s-} \; \diff X^n_s, \qquad t\ge 0. \label{eq:remainder_integrals}
\end{align}

By $M_1$ tightness of the $X^n$ and $H^n$, choose $\delta,R>0$ such that $\Pro^n(A_{n,k}^\complement)\le \eta$
for all $n\ge 1$, and $\ell\ge 1$ large enough that $2|\Theta_\ell|<\delta$, with
\begin{align} 
    A_{n,k} \; := \; \{w'(X^n,\delta)\vee w'(H^n,\delta)  \le \gamma/2\} \cap \{N_{a_k\wedge \gamma}^T(H^n) \vee |X^n|^*_{T} \vee |H^n|^*_T \le R\}, \label{eq:auxiliary_set_probability_bound_proof_remainder_integrals} 
\end{align}
where
\begin{align*} N^T_{c}(x) \; := \;   \operatorname{sup} \{ \, n \,  : \, 0=t_1\le t_2 \le ...\le t_{2n}=T \; , \;  |x(t_{2i})- x(t_{2i-1})|\ge c\}\end{align*}
denotes the maximal number of increments of the càdlàg path $x$ on $[0,T]$ larger than $c>0$ in size. The tightness of $N^T_c(H^n)$, $|X^n|^*_T$, and $|H^n|^*_T$ (directly) follows from the $M_1$ tightness of the $H^n$ and $X^n$ (see \cite[Cor.~A.9]{andreasfabrice_theorypaper}). We shall point out that $\{\#\{j: \tau^{n,k,\ell}_j \le T\} > R\} \subseteq \{N^T_{a_k}(H^n) >R\}$, and therefore there are at most $R$ stopping times $\tau^{n,k,\ell}_j$, $j\ge 1$, with $\tau^{n,k,\ell}_j\le T$ on the set $A_{n,k}$.

Towards the next lemma, recall the definition of $w'$ in \eqref{it2:tightness_of_corrected_integrands_1} of the proof of Lemma \ref{lem:tightness_of_corrected_integrands_1}. By accessing the 'future' of the dynamics, we will be able to closely approximate sequences of real-valued càdlàg processes, which are tight in $M_1$, on small enough intervals through monotone step functions. Lemma \ref{lem:approx_of_M1_by monotone_pieces_X} formalises this observation for deterministic paths. Indeed, it is in fact inherent to relative compactness in $M_1$ that, on short intervals, the paths do uniformly behave "almost" monotonously up to a margin of error, and so this should not come as a surprise. For a deeper discussion on the $M_1$ topology and its underlying geometric interpretation, we refer the reader to \cite[Sec.~12.3]{whitt}.

\begin{lemma} \label{lem:approx_of_M1_by monotone_pieces_X}
    Let $x \in \D_{\R}[0,\infty)$, $\gamma,\delta >0$, $T>0$ and $0\le t_1 < t_2 \le T$ with $t_2-t_1\le \delta$. If $w'(x,\delta)< \gamma/2$, then there exists an increasing càdlàg step function $\xi:[t_1,t_2] \to [0,1]$ such that $\xi(t_1)=0$, $\xi(t_2)=1$, and 
    $$ \big|x(t)- \big(x(t_1)+ \xi(t)(x(t_2)-x(t_1))\big)\big| \; \le \; \gamma$$
    for all $t \in [t_1,t_2]$.
\end{lemma}

The approximation in Lemma \ref{lem:approx_of_M1_by monotone_pieces_X} will generally suspend adaptedness if we apply it to adapted càdlàg stochastic processes. Lemma \ref{lem:approx_of_M1_by monotone_pieces_H} offers an obvious direct alternative for monotone approximation that preserves adaptedness. 

\begin{lemma} \label{lem:approx_of_M1_by monotone_pieces_H}
    Let $x \in \D_{\R}[0,\infty)$, $R,\gamma,\delta >0$, $T>0$ and $0\le t_1 < t_2 \le T$ with $t_2-t_1\le \delta$. If $w'(x,\delta)< \gamma/2$ and $N^T_{\gamma}(x)<R$, then there exists a monotone càdlàg step function $\xi$ of the form
    \begin{align} \xi(t) \; := \; \sum_{j=1}^R x(\sigma_j) \ind_{[\sigma_j\, \wedge \, t_2, \, \sigma_{j+1}\, \wedge \, t_2)}(t), \qquad t_1 \le t < t_2 \label{eq:lem:approx_of_M1_by monotone_pieces_H}
    \end{align}
    as well as $\xi(t_2):= \lim_{t\uparrow t_2} \xi(t)$, where the $\sigma_j$ are stopping times and $|x(t)- \xi(t)| \; \le \; \gamma$ for all $t \in [t_1,t_2]$.
\end{lemma}
The proof of both Lemma \ref{lem:approx_of_M1_by monotone_pieces_X} and \ref{lem:approx_of_M1_by monotone_pieces_H} can be found in the Appendix. 

Whenever there is no risk of confusion, for the rest of this section we will simply write $\lfloor \tau^{n,k,\ell}_j\rfloor$ and $\lceil \tau^{n,k,\ell}_j\rceil$ to denote $\lfloor \tau^{n,k,\ell}_j\rfloor_{\Theta_\ell}$ and $\lceil \tau^{n,k,\ell}_j\rceil_{\Theta_\ell}$. Lemma \ref{lem:approx_of_M1_by monotone_pieces_X} allows us to decompose 
\begin{align}
    X^n_t \; = \; X^n_{\lfloor \tau^{n,k,\ell}_j\rfloor} \, + \, \xi^{j,n,k,\ell}_t(X^n_{\lceil \tau^{n,k,\ell}_j\rceil}- X^n_{\lfloor \tau^{n,k,\ell}_j\rfloor}) \, + \, \phi^{j,n,k,\ell}_t \label{decomp:X^n_for_remainder_integrals}
\end{align} 
for $\lfloor \tau^{n,k,\ell}_j\rfloor \le t \le \lceil \tau^{n,k,\ell}_j\rceil$ on $\{w'(X^n,\delta)\le \gamma/2\}$ (and thus on $A_{n,k})$, where $\xi^{j,n,k,\ell}$ is a coordinatewise non-decreasing, $\R^d$-valued step function with $|\xi^{j,n,k,\ell}|\le d$, and $|\phi^{j,n,k,\ell}|\le d\gamma$. Recall that $d$ is the dimension of the random vectors $H^n,X^n$. We shall emphasise that neither $\xi^{j,n,k,\ell}$ nor $\phi^{j,n,k,\ell}$ are generally adapted. This will not pose an issue in the remaining parts of this section, as will be thoroughly explained at the relevant positions. Similarly, we can decompose the $H^n$ on $A_{n.k}$ according to Lemma \ref{lem:approx_of_M1_by monotone_pieces_X} by 
\begin{align}
    \tilde{H}^{n,k,\ell}_t &= H^n_t  - H^n_{\tau^{n,k,\ell}_j-}  = ( H^n_{\lfloor \tau^{n,k,\ell}_j\rfloor} - H^n_{\tau^{n,k,\ell}_j-})  + \, \zeta^{j,n,k,\ell}_t(H^n_{\lceil \tau^{n,k,\ell}_j\rceil}- H^n_{\lfloor \tau^{n,k,\ell}_j\rfloor}) \, + \, \psi^{j,n,k,\ell}_t   
    \label{decomp:H^n_for_remainder_integrals_non_adapted}
\end{align} 
for $\tau^{n,k,\ell}_j \le t \le \lceil \tau^{n,k,\ell}_j\rceil$, where as above $\zeta^{j,n,k,\ell}$ is a coordinatewise non-decreasing, $\R^d$-valued step function with $|\zeta^{j,n,k,\ell}|\le d$, and $|\psi^{j,n,k,\ell}|\le d\gamma$. Moreover, we obtain a second decomposition from Lemma \ref{lem:approx_of_M1_by monotone_pieces_H}, 
\begin{align}
    \tilde{H}^{n,k,\ell}_t \; &= \; \hat{\zeta}^{j,n,k,\ell}_t \, + \, \hat\psi^{j,n,k,\ell}_t  
    \label{decomp:H^n_for_remainder_integrals_adapted}
\end{align} 
for $\tau^{n,k,\ell}_j \le t \le \lceil \tau^{n,k,\ell}_j\rceil$, where $\hat{\zeta}^{j,n,k,\ell}$ is a coordinatewise monotone and càdlàg simple $\R^d$-valued process with $|\hat{\zeta}^{j,n,k,\ell}_t|\le |H^n|^*_T$, and $|\hat\psi^{j,n,k,\ell}|\le d\gamma$. We point out that---unlike the decompositions in \eqref{decomp:X^n_for_remainder_integrals}, \eqref{decomp:H^n_for_remainder_integrals_non_adapted}---the processes of the decomposition \eqref{decomp:H^n_for_remainder_integrals_adapted} are in fact adapted processes.

In a next step, we are going to use the previous decompositions in order to decompose the remainder integrals on the right-hand side of \eqref{eq:remainder_integrals}. More precisely, on $A_{n,k}$, \eqref{decomp:X^n_for_remainder_integrals} gives us 
\begin{align*} \int_{t \wedge \tau^{n,k,\ell}_j }^{ t \wedge \rho^{n,k,\ell}_j} &\tilde{H}^{n,k,\ell}_{s-} \; \diff X^n_s \; = \; \\[-2ex]
& (X^n_{\lceil \tau^{n,k,\ell}_j\rceil}- X^n_{\lfloor \tau^{n,k,\ell}_j\rfloor})\int_{t \wedge \tau^{n,k,\ell}_j }^{ t \wedge \rho^{n,k,\ell}_j} \tilde{H}^{n,k,\ell}_{s-} \, \diff \xi^{j,n,k,\ell}_s \, + \, \int_{t \wedge \tau^{n,k,\ell}_j }^{ t \wedge \rho^{n,k,\ell}_j} \tilde{H}^{n,k,\ell}_{s-} \, \diff \phi^{j,n,k,\ell}_s
\end{align*}
for each $j,n,k,\ell \ge 1$ and $t\ge 0$, where we understand the integral with respect to $\diff \phi^{j,n,k,\ell}$ as the difference of two integrals with respect to $\diff X^n$ and $\diff (X^n_{t_1}+ \phi^{j,n,k,\ell}(X^n_{t_2}-X^n_{t_1}))$ respectively. Now, applying decomposition \eqref{decomp:H^n_for_remainder_integrals_non_adapted} for the integrands of the first integral and decomposition \eqref{decomp:H^n_for_remainder_integrals_adapted} for the integrands of the second integral, we obtain
\begin{align}
   \int_{t \wedge \tau^{n,k,\ell}_j }^{ t \wedge \rho^{n,k,\ell}_j} &\tilde{H}^{n,k,\ell}_{s-}  \diff X^n_s  = (H^n_{\lceil \tau^{n,k,\ell}_j\rceil}- H^n_{\lfloor \tau^{n,k,\ell}_j\rfloor}) (X^n_{\lceil \tau^{n,k,\ell}_j\rceil}- X^n_{\lfloor \tau^{n,k,\ell}_j\rfloor}) \int_{t \wedge \tau^{n,k,\ell}_j }^{ t \wedge \rho^{n,k,\ell}_j} \zeta^{j,n,k,\ell}_{s-}  \diff \xi^{j,n,k,\ell}_s \notag \\
   &+\; (H^n_{\lfloor \tau^{n,k,\ell}_j\rfloor}-H^n_{\tau^{n,k,\ell}_j-} )( \xi^{j,n,k,\ell}_{t\wedge \lceil \tau^{n,k,\ell}_j \rceil }-\xi^{j,n,k,\ell}_{t\wedge \tau^{n,k,\ell}_j}) \, + \, \int_{t \wedge \tau^{n,k,\ell}_j }^{ t \wedge \rho^{n,k,\ell}_j} \psi^{j,n,k,\ell}_{s-}  \diff \xi^{j,n,k,\ell}_s \notag \\
    &+\; \int_{t \wedge \tau^{n,k,\ell}_j }^{ t \wedge \rho^{n,k,\ell}_j} \hat \zeta^{j,n,k,\ell}_{s-}\, \diff \phi^{j,n,k,\ell}_s \, + \, \int_{t \wedge \tau^{n,k,\ell}_j }^{ t \wedge \rho^{n,k,\ell}_j} \hat\psi^{j,n,k,\ell}_{s-}  \diff \phi^{j,n,k,\ell}_s\label{decomp:remainder_integral}
\end{align}
on $A_{n,k}$.

We will say that a family of processes $Z^{n,k,\ell}$ \emph{vanishes conveniently in probability}, if for every $\varepsilon>0$ there is
$$ \lim\limits_{\gamma \to 0} \, \limsup\limits_{k\to \infty} \, \limsup\limits_{\ell \to \infty} \, \limsup\limits_{n\to \infty} \, \Pro^n( |Z^{n,k,\ell}|^*_T \, > \, \varepsilon, \, A_{n,k}) \; = \; 0.$$
With this definition at hand, we state the following lemma.

\begin{lemma} \label{lem:junk_vanishes_conveniently}
    All but the first summand on the right-hand side of \eqref{decomp:remainder_integral} vanish conveniently in probability.
\end{lemma}
\begin{proof}
    It suffices to show the claim individually for each of the four processes. On $A_{n,k}$, the following bounds can be derived.
    \begin{enumerate}
        \item By definition of $\tau^{n,k,\ell}_j$, the fact that $\rho^{n,k,\ell}_{j-1}\vee (\tau^{n,k,\ell}_j-|\Theta_\ell|) \le \lfloor \tau^{n,k,\ell}_j\rfloor\le \tau^{n,k,\ell}_j$, and $|\xi^{j,n,k,\ell}|\le d$, the second summand can be bounded by
        $$ |(H^n_{\lfloor \tau^{n,k,\ell}_j\rfloor}-H^n_{\tau^{n,k,\ell}_j-} )( \xi^{j,n,k,\ell}_{t\wedge \lceil \tau^{n,k,\ell}_j \rceil }-\xi^{j,n,k,\ell}_{t\wedge \tau^{n,k,\ell}_j})| \; \le \; 2d\, a_k$$
        for all $t\ge 0$, $n,k,\ell \ge 1$.
        \item \label{it2:control_the_junk} The $\xi^{j,n,k,\ell}$ are coordinatewise increasing and therefore processes of finite variation. Hence, the third summand is a Lebesgue-Stieltjes integral, and it holds that 
        $$\Big| \int_{t \wedge \tau^{n,k,\ell}_j }^{ t \wedge \rho^{n,k,\ell}_j} \psi^{j,n,k,\ell}_{s-} \, \diff \xi^{j,n,k,\ell}_s \Big| \; \le \; \int_{t \wedge \tau^{n,k,\ell}_j }^{ t \wedge \rho^{n,k,\ell}_j} |\psi^{j,n,k,\ell}_{s-}| \, \diff \operatorname{TV}_{[0,s]}(\xi^{j,n,k,\ell})\; \le \; d^2\gamma $$
        for all $t\ge 0$, $n,k,\ell \ge 1$, due to $|\psi^{j,n,k,\ell}|\le d\gamma$, and $\xi^{j,n,k,\ell}\in [0,1]^d$ being coordinatewise increasing on $[\tau^{n,k,\ell}_j,\rho^{n,k,\ell}_j]$.
        \item With regard to the fourth summand, recall that $\hat{\zeta}^{j,n,k,\ell}$ is almost surely a càdlàg simple process, which can be written as 
        $$\hat{\zeta}^{j,n,k,\ell}_t=\sum_{i=1}^R \hat{\zeta}^{j,n,k,\ell}_{\sigma_i} \ind_{[\sigma_i\, \wedge \, \rho^{n,k,\ell}_j, \, \sigma_{i+1}\, \wedge\,  \rho^{n,k,\ell}_j)}(t)$$
        for stopping times $\sigma_i$ with $\sigma_i \le \sigma_{i+1}$ for all $i\ge 1$, and every $t \in [\tau^{n,k,\ell}_j,\rho^{n,k,\ell}_j]$. Further recall that $|\hat{\zeta}^{j,n,k,\ell}_t|\le |H^n|^*_T\le R$ on $A_{n,k}$. Without loss of generality, assume that $\sigma_1=\tau^{n,k,\ell}_j$. Then, the fourth summand on the right-hand side of \eqref{decomp:remainder_integral} is a simple integral, which by restructuring of the sum (i.e. integration by parts) becomes
        \begin{align*}
           \int_{t \wedge \tau^{n,k,\ell}_j }^{ t \wedge \rho^{n,k,\ell}_j} \hat \zeta^{j,n,k,\ell}_{s-}\, \diff \phi^{j,n,k,\ell}_s \; &= \; \sum_{i=1}^{R-1} \hat{\zeta}^{j,n,k,\ell}_{\sigma_i} (\phi^{j,n,k,\ell}_{\sigma_{i+1}\, \wedge\, t\,\wedge\,  \rho^{n,k,\ell}_j}-\phi^{j,n,k,\ell}_{\sigma_{i}\, \wedge\, t\,\wedge\,  \rho^{n,k,\ell}_j}) \\
           &= \; \phi^{j,n,k,\ell}_{\sigma_R \, \wedge \, t\,\wedge\,  \rho^{n,k,\ell}_j}\, \hat{\zeta}^{j,n,k,\ell}_{\sigma_{R-1}} \; - \; \phi^{j,n,k,\ell}_{\sigma_1 \, \wedge \, t\,\wedge\,  \rho^{n,k,\ell}_j}\, \hat{\zeta}^{j,n,k,\ell}_{\sigma_1} \\
           & \qquad + \sum_{i=2}^{R-1} \phi^{j,n,k,\ell}_{\sigma_i \, \wedge \, t\,\wedge\,  \rho^{n,k,\ell}_j} (\hat\zeta^{j,n,k,\ell}_{\sigma_{i-1}}-\hat\zeta^{j,n,k,\ell}_{\sigma_{i}}).
        \end{align*}
        Using that $|\hat{\zeta}^{j,n,k,\ell}_t|\le R$, that $\hat\zeta^{j,n,k,\ell}$ is coordinatewise monotone, and $|\phi^{j,n,k,\ell}|\le d\gamma$, we deduce
        $$ \Big|\int_{t \wedge \tau^{n,k,\ell}_j }^{ t \wedge \rho^{n,k,\ell}_j} \hat \zeta^{j,n,k,\ell}_{s-}\, \diff \phi^{j,n,k,\ell}_s\Big| \; \le \; 3Rd\gamma$$
        for all $t\ge 0$, $n,k,\ell \ge 1$.
        \item For the last summand, we recall that $\phi^{j,n,k,\ell}=X^n-\xi^{j,n,k,\ell}$, and the $\xi^{j,n,k,\ell} \in [0,1]^d$ are coordinatewise non-decreasing and therefore of finite variation. Hence, we can rewrite the integrals as 
        \begin{align}
            \int_{t \wedge \tau^{n,k,\ell}_j }^{ t \wedge \rho^{n,k,\ell}_j} \hat\psi^{j,n,k,\ell}_{s-} \, \diff \phi^{j,n,k,\ell} \; = \; \int_{t \wedge \tau^{n,k,\ell}_j }^{ t \wedge \rho^{n,k,\ell}_j} \hat\psi^{j,n,k,\ell}_{s-} \, \diff X^n_s \; - \; \int_{t \wedge \tau^{n,k,\ell}_j }^{ t \wedge \rho^{n,k,\ell}_j} \hat\psi^{j,n,k,\ell}_{s-} \, \diff \xi^{j,n,k,\ell}_s. \label{eq:junk_converging_to_zero_1}
        \end{align} 
        The second integral on the right-hand side of \eqref{eq:junk_converging_to_zero_1} can be estimated as in point \ref{it2:control_the_junk} above. Turning to the first integral on the right-hand side of \eqref{eq:junk_converging_to_zero_1}, we will use the good decompositions \eqref{eq:Mn_An_condition} of $X^n$, together with the adaptedness of $\hat\psi^{j,n,k,\ell}$ (which makes the integral process a local martingale) and the fact that $|\hat\psi^{j,n,k,\ell}|\le d\gamma$ by construction. Applying Lenglart's inequality \cite[Lem.~I.3.30]{shiryaev} with the $L$-domination property satisfied due to the classical Burkholder--Davis--Gundy inequality, we obtain
        \begin{align*}
            \Pro^n &\Big( \Big| \int_{\bullet \wedge \tau^{n,k,\ell}_j }^{ \bullet \wedge \rho^{n,k,\ell}_j} \hat\psi^{j,n,k,\ell}_{s-} \, \diff M^n_s\Big|^*_T  >  \varepsilon\Big)\\
            &\le \; \frac{\Lambda d\gamma ( \sqrt{\beta}  + \E^n[\, |\Delta M^n_{T\wedge \sigma_1\wedge \sigma_2}|\, ])}{\varepsilon} \; + \; \Pro^n( |M^n|^*_T  \ge  R )\;  + \; \Pro^n( [M^n]_T  \ge  \beta  ) 
        \end{align*}
        for any $\beta>0$, where $\sigma_1:=\operatorname{inf}\{t\ge \tau^{n,k,\ell}_j : [M^n]_T\ge \beta\}$, $\sigma_2:=\operatorname{inf}\{t\ge \tau^{n,k,\ell}_j : |M^n|^*_T\ge R\}$ and $\Lambda$ denotes the 'upper' constant arising from the Burkholder-Davis-Gundy inequality. Another application of Lenglart's inequality to $\Pro^n( [M^n]_T  \ge  \beta  )$ finally, yields 
        \begin{align}
            \Pro^n \Big( \Big| \int_{\bullet \wedge \tau^{n,k,\ell}_j }^{ \bullet \wedge \rho^{n,k,\ell}_j} \hat\psi^{j,n,k,\ell}_{s-} \, &\diff M^n_s\Big|^*_T  > \varepsilon \Big) \; \le \; \frac{\Lambda d\gamma ( \sqrt{\beta} + 2R +  \E^n[\, |\Delta M^n_{T\wedge \sigma_2}|\, ])}{\varepsilon} \notag \\
            &+ \; \frac{\lambda ( 2R  + \E^n[\, |\Delta M^n_{T\wedge \sigma_2}|\, ])}{\beta}\; + \; 2\Pro^n(|M^n|^*_T \ge  2R ) \label{eq:junk_converging_to_zero_2}
        \end{align}
        for any $\beta>0$, where $\lambda>0$ denotes the 'lower' constant arising from the Burkholder-Davis-Gundy inequality. Due to \eqref{eq:Mn_An_condition}, without loss of generality assume that $R>0$ is such that $\limsup
        _{n\to \infty}\Pro^n(\operatorname{TV}_{[0,T]}(A^n)>R)\le \gamma$, and denote 
        $\Gamma  := \limsup_{n\to \infty}\, \E^n[\, |\Delta M^n_{T\wedge \sigma_2}|\, ]  < \infty$.
        Since $R>0$ was also chosen so that we have $\operatorname{sup}_{n\ge 1}\Pro^n(|X^n|^*_{T}> R) \le \gamma$, the former implies 
        $$\limsup\limits_{n\ge 1}\Pro^n(|M^n|^*_{T}> 2R) \; \le \;  \operatorname{sup}_{n\ge 1}\Pro^n(|X^n|^*_{T}> R)\; + \; \limsup\limits_{n\ge 1}\Pro^n(\operatorname{TV}_{[0,T]}(A^n)>R) $$
        due to $M^n=X^n-A^n$, which is bounded by $2\gamma$. Consequently, by choosing $\beta:=1/\gamma$, \eqref{eq:junk_converging_to_zero_2} yields
        \begin{align*}
           \limsup\limits_{n\to \infty} \, \Pro^n &\Big( \Big| \int_{\bullet \wedge \tau^{n,k,\ell}_j }^{ \bullet \wedge \rho^{n,k,\ell}_j} \hat\psi^{j,n,k,\ell}_{s-} \, \diff M^n_s\Big|^*_T  > \varepsilon \Big) \; \le \; \frac{d\Lambda}{\varepsilon} \sqrt{\gamma} \, + \,\frac{(d\Lambda+\lambda)(2R+\Gamma)}{\varepsilon} \gamma \, + \, 4\gamma 
        \end{align*}
        which tends to zero as $\gamma\to 0$. More directly, since the $A^n$ are of finite variation, we also obtain
        $$ \limsup\limits_{n\to \infty} \, \Pro^n \Big( \Big| \int_{\bullet \wedge \tau^{n,k,\ell}_j }^{ \bullet \wedge \rho^{n,k,\ell}_j} \hat\psi^{j,n,k,\ell}_{s-} \, \diff A^n_s\Big|^*_T  > \varepsilon \Big) \; \le \; \limsup\limits_{n\to \infty} \, \Pro^n \Big(  \operatorname{TV}_{[0,T]}(A^n)  > \varepsilon/\gamma\Big) \to 0$$
        as $\gamma \to 0$.
    \end{enumerate}
\end{proof}

Turning to the investigation of the limiting behaviour of the first summand of \eqref{decomp:remainder_integral}, we begin with the following lemma. For a càdlàg process $\tilde{H}^0$, we will define $\tilde{\tau}^{0,k,\ell}_j$ in full analogy to the quantities $\tau^{0,k,\ell}_j$ before \eqref{eq:definition_correction_processes}, simply with $H^0$ replaced by $\tilde{H}^0$.

\begin{lemma} \label{lem:conv_relevant_quant_remainder_integrals}
    Suppose $(H^n,X^n)\Rightarrow (\tilde{H}^0,\tilde{X}^0)$ on $(\D_{\R^d}[0,\infty),\dM)^2$. Then, for all $k,\ell\ge 1$ there is weak convergence
    \begin{align}
        &\Big(H^n, X^n, \big(\tau^{n,k,\ell}_j, \lfloor \tau^{n,k,\ell}_j\rfloor, \lceil \tau^{n,k,\ell}_j\rceil, H^n_{\lfloor \tau^{n,k,\ell}_j\rfloor}, H^n_{\lceil \tau^{n,k,\ell}_j\rceil},X^n_{\lfloor \tau^{n,k,\ell}_j\rfloor},X^n_{\lceil \tau^{n,k,\ell}_j\rceil}\big)_{j\ge 1} \Big) \notag \\[1ex]
        & \qquad \Rightarrow \Big(\tilde{H}^0, \tilde{X}^0, \big(\tilde{\tau}^{0,k,\ell}_j, \lfloor \tilde{\tau}^{0,k,\ell}_j\rfloor, \lceil \tilde{\tau}^{0,k,\ell}_j\rceil, \tilde{H}^0_{\lfloor \tilde{\tau}^{0,k,\ell}_j\rfloor}, \tilde{H}^0_{\lceil \tilde{\tau}^{0,k,\ell}_j\rceil},\tilde{X}^0_{\lfloor \tilde{\tau}^{0,k,\ell}_j\rfloor},\tilde{X}^0_{\lceil \tilde{\tau}^{0,k,\ell}_j\rceil}\big)_{j\ge 1} \Big) \label{eq:conv_relevant_quant_remainder_integrals}
    \end{align}
    on $(\D_{\R^d}[0,\infty),\dM)^2 \times ([0,\infty)^3 \times \R^{4d},|\cdot|)^\N$ as $n\to \infty$.
\end{lemma}

The proof of Lemma \ref{lem:conv_relevant_quant_remainder_integrals} involves additional technicalities and requires a specific choice of the sequence ${a_k}$, while adding little to the main line of argument. We therefore defer the proof to the Appendix, where it appears after Lemma \ref{lem:prerequisite_final_suitable_sequence_of_large_increments_stopping_time_defi}.

Toward the next proposition, let us define 
\begin{align} Y^{j,n,k,\ell}_t \; := \; \int_{t \wedge \tau^{n,k,\ell}_j }^{ t \wedge \rho^{n,k,\ell}_j} \zeta^{j,n,k,\ell}_{s-} \, \diff \xi^{j,n,k,\ell}_s  \label{eq:defi_scaling_terms_Y_n}
\end{align}
for each $j,n,k,\ell \ge 1$ and $t\ge 0$.
\begin{proposition} \label{prop:conv_remainder_integrals_2}
    Let $c>0$. There exists a subsequence of $\{(H^n,X^n)\}$ (which, for conciseness of notation, we will denote as the original sequence) and a probability space $(\tilde{\Omega}^0, \tilde{\mathcal{F}}^0, \tilde{\Pro}^0)$ such that 
        \begin{align}
        &\Big(H^n, X^n, \big(\tau^{n,k,\ell}_j, \lfloor \tau^{n,k,\ell}_j\rfloor, \lceil \tau^{n,k,\ell}_j\rceil, H^n_{\lfloor \tau^{n,k,\ell}_j\rfloor}, H^n_{\lceil \tau^{n,k,\ell}_j\rceil},X^n_{\lfloor \tau^{n,k,\ell}_j\rfloor},X^n_{\lceil \tau^{n,k,\ell}_j\rceil}, Y^{j,n,k,\ell} \big)_{j,k,\ell\ge 1} \Big) \notag \\[1ex]
        & \quad \Rightarrow \Big(\tilde{H}^0, \tilde{X}^0, \big(\tilde{\tau}^{0,k,\ell}_j, \lfloor \tilde{\tau}^{0,k,\ell}_j\rfloor, \lceil \tilde{\tau}^{0,k,\ell}_j\rceil, \tilde{H}^0_{\lfloor \tilde{\tau}^{0,k,\ell}_j\rfloor}, \tilde{H}^0_{\lceil \tilde{\tau}^{0,k,\ell}_j\rceil},\tilde{X}^0_{\lfloor \tilde{\tau}^{0,k,\ell}_j\rfloor},\tilde{X}^0_{\lceil \tilde{\tau}^{0,k,\ell}_j\rceil},\tilde{Y}^{j,0,k,\ell}\big)_{j,k,\ell\ge 1} \Big) \label{eq:conv_relevant_quant_remainder_integrals_2}
    \end{align}
    on $(\D_{\R^d}[0,\infty),\dM)^2 \times \big(([0,\infty)^3 \times \R^{4d},|\cdot|)\times (\D_{\R^d}[-c,\infty),\dM)\big)^\N$ as $n\to \infty$, where the limiting quantities are all defined on $(\tilde{\Omega}^0, \tilde{\mathcal{F}}^0,  \tilde{\Pro}^0)$. In particular, in this case there is $\mathcal{L} (\tilde H^{0}, \tilde X^{0})= \mathcal{L} (H^0, X^0)$.
\end{proposition}
\begin{proof}
For each $j,k,\ell\ge 1$, the processes $Y^{j,n,k,\ell}$, $n\ge 1$, are coordinatewise non-decreasing and, by construction of $\zeta^{j,n,k,\ell}, \xi^{j,n,k,\ell}$, almost surely bounded by one in each component. By naturally extending the processes $Y^{j,n,k,\ell}$ to $[-c,\infty)$ for some $c>0$, each component is tight in $(\D_{\R}[-c,\infty),\dM)$. Indeed, the $M_1$ tightness criteria \eqref{it1:tightness_of_corrected_integrands_1}-\eqref{it3:tightness_of_corrected_integrands_1} in the proof of Lemma \ref{lem:tightness_of_corrected_integrands_1} can be seen to be immediately satisfied, where the local uniform convergence at zero is simply replaced by local uniform convergence at $-c$. Now, coordinatewise tightness is sufficient for tightness of the $Y^{j,n,k,\ell}$ on the product space $(\D_{\R}[-c,\infty),\dM)^d$, due to the characterisation of compact sets in product Hausdorff spaces as precisely those closed sets which are contained in the Cartesian product of compact sets of the individual spaces. In this case, that is $Q$ is compact in $(\D_{\R}[-c,\infty), \dM)^d$ if and only if $Q$ is closed and $Q \subseteq K_1\times ...\times K_d$ with each $K_i$ compact in $(\D_{\R}[-c,\infty), \dM)$. Repeating this argument, we obtain the tightness of the sequence 
\begin{align*}
    \mathfrak{S}^{n}_L \; := \; \Big(H^n, \,X^n, \,\big(\tau^{n,k,\ell}_j,\, &\lfloor \tau^{n,k,\ell}_j\rfloor, \,\lceil \tau^{n,k,\ell}_j\rceil, \,H^n_{\lfloor \tau^{n,k,\ell}_j\rfloor}, \,H^n_{\lceil \tau^{n,k,\ell}_j\rceil},\\
    &X^n_{\lfloor \tau^{n,k,\ell}_j\rfloor},\,X^n_{\lceil \tau^{n,k,\ell}_j\rceil}, \, Y^{j,n,k,\ell} \big)_{j,\ell,k= 1,...,L} \Big), \qquad n\ge 1,
\end{align*}
on the respective product spaces $$\mathcal{S}_L:=(\D_{\R^d}[0,\infty),\dM)^2 \times \big(([0,\infty)^3 \times \R^{4d},|\cdot|)\times (\D_{\R}[-c,\infty),\dM)^d\big)^{L^3}$$ endowed with their product Borel $\sigma$-algebras. By Prokhorov's Theorem, to $\{\mathfrak{S}^{n}_L\}_{n\ge 1}$ there exists a subsequence (which we will denote as the original sequence) such that the law of $\mathfrak{S}^{n}_L$ converges weakly, that is
$$ \mathcal{L}(\mathfrak{S}^{n}_L) \; \Rightarrow \; \mu_L$$
as $n\to \infty$. It is straightforward to verify that the family $\{\mu_L\}$ satisfies the assumptions of the Kolmogorov extension theorem, for example by the continuous mapping theorem. Hence, there exists a probability space $(\tilde{\Omega}^0, \tilde{\mathcal{F}}^0, \tilde\Pro^0)$ and tuples
\begin{align*}
    \tilde{\mathfrak{S}}^{0}  \hspace{-0.5ex}:= \hspace{-0.5ex}\Big(\tilde{H}^0, \tilde{X}^0, \big(\tilde{\tau}^{0,k,\ell}_j, \lfloor \tilde{\tau}^{0,k,\ell}_j\rfloor, \lceil \tilde{\tau}^{0,k,\ell}_j\rceil, \tilde{H}^0_{\lfloor \tilde{\tau}^{0,k,\ell}_j\rfloor}, \tilde{H}^0_{\lceil \tilde{\tau}^{0,k,\ell}_j\rceil},\tilde{X}^0_{\lfloor \tilde{\tau}^{0,k,\ell}_j\rfloor},\tilde{X}^0_{\lceil \tilde{\tau}^{0,k,\ell}_j\rceil},\tilde{Y}^{j,0,k,\ell}\big)_{j,k,\ell\ge 1} \Big)
\end{align*}
on this space taking values in the product space
$$ (\D_{\R^d}[0,\infty),\dM)^2 \times \big(([0,\infty)^3 \times \R^{4d},|\cdot|)\times (\D_{\R}[-c,\infty),\dM)^d\big)^{\N} $$
endowed with its product Borel $\sigma$-algebra, such that the pushforward measure through the restrictions of $\mathcal{L}(\tilde{\mathfrak{S}}^{0})$ to the respective first $2+8L^3$ components coincides with $\mu_L$.
Indeed, the concrete specified form of the tuple $\tilde{\mathfrak{S}}^{0}$---except for the quantities $\tilde{Y}^{j,0,k,\ell}$---follows directly from the continuity of coordinate restrictions, the continuous mapping theorem \cite[Thm.~2.7]{billingsley} and Lemma \ref{lem:conv_relevant_quant_remainder_integrals}. Moreover, we clearly have $\mathcal{L} (\tilde H^{0}, \tilde X^{0})=\mathcal{L} (H^0, X^0)$. As per Lemma \ref{lem:properties_of_weighing_limit_terms} below, the limiting $\tilde{Y}^{j,0,k,\ell}$ are coordinatewise non-decreasing, hence yielding their convergence in the stronger space $(\D_{\R^d}[-c,\infty),\dM)$ (see e.g. \cite[Thm.~12.7.3]{whitt}).
\end{proof}

We will collect a crucial property of the processes $\tilde{Y}^{j,0,k,\ell}$ that arise in the limit of \eqref{eq:conv_relevant_quant_remainder_integrals_2}.
\begin{lemma} \label{lem:properties_of_weighing_limit_terms} Under the notation of Proposition \ref{prop:conv_remainder_integrals_2}, it holds that for all $j,k,\ell \ge 1$, 
    $$ \tilde\Pro^0\Big( \tilde{Y}^{j,0,k,\ell} = 0 \text{ on } [0,\tilde{\tau}^{0,k,\ell}_j) \text{ and } \tilde{Y}^{j,0,k,\ell} \in [0,1]^d \text{ is constant on } (\lceil \tilde{\tau}^{0,k,\ell}_j\rceil,\infty) \Big) \; = \; 1 $$
    and the process $\tilde{Y}^{j,0,k,\ell}$ is almost surely coordinatewise non-decreasing.
\end{lemma}
\begin{proof}
    Fix $j,k,\ell \ge 1$ and define the set 
    $$ B \; := \; \union{t > 0}{} \Big( \{ \alpha \in \D_{\R^d}[0,\infty): \exists s \in [0,t) \text{ s.t. } \alpha(s)\neq 0\} \times (t,\infty) \Big)$$
    which is open in the product topology of $(\D_{\R}[0,\infty), \dM)^d \times ([0,\infty),|\cdot|)$ and clearly
    $$ (Z,\sigma) \in B \text{ iff there exists } s \in [0, \sigma) \text{ with } Z_s \neq 0.$$
    Since $(Y^{j,n,k,\ell}, \tilde{\tau}^{n,k,\ell}_j)\Rightarrow (\tilde{Y}^{j,0,k,\ell}, \tilde{\tau}^{0,k,\ell}_j)$ according to Proposition \ref{prop:conv_remainder_integrals_2}, making use of the Portmanteau theorem yields
    \begin{align*}
        \tilde{\Pro}^0 &(\exists s \in [0, \tilde{\tau}^{0,k,\ell}_j) : \tilde{Y}^{j,0,k,\ell}_s \neq 0) \; \le \; \liminf_{n\to \infty} \; \Pro^{n} \Big(\exists s \in [0, \tau^{n,k,\ell}_j) : Y^{j,n,k,\ell}_s \neq 0\Big) \; = \; 0,
    \end{align*}
    where the last equation follows directly from the definition of the $Y^{j,n,k,\ell}$ in \eqref{eq:defi_scaling_terms_Y_n}.
    One now proceeds in the same way with $(\tilde{Y}^{j,0,k,\ell}, \lceil\tilde{\tau}^{0,k,\ell}_j\rceil )$ for the open set 
    $$ C \; := \; \union{t > 0}{} \Big( \{ \alpha \in \D_{\R^d}[0,\infty): \exists s_1,s_2 \in (t,\infty) \text{ s.t. } \alpha(s_1)\neq \alpha(s_2)\} \times (0,t) \Big)$$
    in order to complete the proof of the first part of the lemma. The monotonicity follows similarly due to the fact that the set of non-decreasing càdlàg paths is a closed subset of the $M_1$ Skorokhod space.
\end{proof}

Recall the definition of stopping times $\{\tilde{\tau}^{0,k}_j\}_{k,j\ge 1}$ and $\mathcal{T}(\tilde{H}^0)=\{\tilde{\sigma}^{0,k}_j\}_{k,j\ge 1}$ from \eqref{eq:def_limiting_stopping_times_tightness_without_AVCI}. 

\begin{lemma} \label{lem:conv_weighing_quant_remainder_integrals_2}
    To every subsequence of $\{(H^n,X^n)\}$ such that there is convergence \eqref{eq:conv_relevant_quant_remainder_integrals_2}, there exists a subsequence of the natural numbers $\{\ell_r\}$ and càdlàg processes $\tilde{Y}^{j,0,k}$, $j,k\ge 1$, on a common probability space such that 
    $\tilde{Y}^{j,0,k,\ell_r} \; \Rightarrow \; \tilde{Y}^{j,0,k}$ as $r\to \infty$ on $(\D_{\R^d}[0,\infty), \dM)$ for all $j,k\ge 1$. Without loss of generality, we may assume the common probability space to be $(\tilde \Omega^0, \tilde{\mathcal{F}}^0, \tilde\Pro^0)$ from Proposition \ref{prop:conv_remainder_integrals_2} and that  
    \begin{align} 
        & \Big(\tilde{H}^0, \tilde{X}^0, \big(\tilde{\tau}^{0,k,\ell_r}_j, \lfloor \tilde{\tau}^{0,k,\ell_r}_j\rfloor, \lceil \tilde{\tau}^{0,k,\ell_r}_j\rceil, \tilde{H}^0_{\lfloor \tilde{\tau}^{0,k,\ell_r}_j\rfloor}, \tilde{H}^0_{\lceil \tilde{\tau}^{0,k,\ell_r}_j\rceil},\tilde{X}^0_{\lfloor \tilde{\tau}^{0,k,\ell_r}_j\rfloor},\tilde{X}^0_{\lceil \tilde{\tau}^{0,k,\ell_r}_j\rceil},\tilde{Y}^{j,0,k,\ell_r}\big)_{j,k\ge 1} \Big) \notag \\
        & \qquad \qquad \Rightarrow \;  \Big(\tilde{H}^0, \tilde{X}^0, \big(\tilde{\tau}^{0,k}_j,  \tilde{\tau}^{0,k}_j, \tilde{\tau}^{0,k}_j, \tilde{H}^0_{\tilde{\tau}^{0,k}_j-}, \tilde{H}^0_{\tilde{\tau}^{0,k}_j},\tilde{X}^0_{\tilde{\tau}^{0,k}_j-},\tilde{X}^0_{ \tilde{\tau}^{0,k}_j},\tilde{Y}^{j,0,k}\big)_{j,k\ge 1} \Big)
    \label{eq:conv_weights_quant_remainder_integral_2}
    \end{align}
    on $(\D_{\R^d}[0,\infty),\dM)^2 \times \big(([0,\infty)^3 \times \R^{4d},|\cdot|)\times (\D_{\R^d}[-c,\infty),\dM)\big)^\N$ as $r \to \infty$. Moreover, each $\tilde{Y}^{j,0,k} \in [0,1]^d$ is a single-jump process and can be written as
    $$ \tilde{Y}^{j,0,k} \; = \; \tilde{Y}^{j,0,k}_{\tilde{\tau}^{0,k}_j} \ind_{[\tilde{\tau}^{0,k}_j,\, \infty)}.$$
    
\end{lemma}
\begin{proof}
    The relative compactness of the left-hand side of \eqref{eq:conv_weights_quant_remainder_integral_2} follows from the relative compactness of the individual component sequences, in the same way as described at the beginning of the proof of Proposition \ref{prop:conv_remainder_integrals_2}. For this, in particular one makes use of the fact that, according to Lemma \ref{lem:properties_of_weighing_limit_terms}, the $\tilde{Y}^{j,0,k,\ell}$ are bounded coordinatewise non-decreasing processes. Hence, there exists a subsequence $\{\ell_r\}$ of the natural numbers such that the left-hand side of \eqref{eq:conv_weights_quant_remainder_integral_2} converges weakly as $r \to \infty$. Denote the individual limits of the $\tilde{Y}^{j,0,k,\ell_r}$ by $\tilde{Y}^{j,0,k}$. By definition of the $\tilde{\tau}^{0,k,\ell}_j$ and the paths of $\tilde{H}^0$ being càdlàg, it becomes clear that almost surely $\tilde{\tau}^{0,k,\ell}_j=\tilde{\tau}^{0,k}_j$ for every $j,k\ge 1$ and large enough $\ell$. This directly entails $\lceil \tilde{\tau}^{0,k,\ell}_j \rceil \downarrow \tilde{\tau}^{0,k}_j$ and $\lfloor \tilde{\tau}^{0,k,\ell}_j \rfloor \uparrow \tilde{\tau}^{0,k}_j$ almost surely as $\ell \to \infty$. The processes $\tilde{H}^0$, $\tilde{X}^0$ being càdlàg, this also implies the almost sure convergence $(\tilde{H}^0_{\lfloor \tilde{\tau}^{0,k,\ell_r}_j\rfloor}, \tilde{H}^0_{\lceil \tilde{\tau}^{0,k,\ell_r}_j\rceil},\tilde{X}^0_{\lfloor \tilde{\tau}^{0,k,\ell_r}_j\rfloor},\tilde{X}^0_{\lceil \tilde{\tau}^{0,k,\ell_r}_j\rceil}) \to (\tilde{H}^0_{\tilde{\tau}^{0,k}_j-}, \tilde{H}^0_{\tilde{\tau}^{0,k}_j},\tilde{X}^0_{\tilde{\tau}^{0,k}_j-},\tilde{X}^0_{ \tilde{\tau}^{0,k}_j})$.
    Without loss of generality, we may assume the common limiting probability space to be $(\tilde \Omega^0, \tilde{\mathcal{F}}^0, \tilde\Pro^0)$ from Proposition \ref{prop:conv_remainder_integrals_2}, and denote all limiting quantities as on the right-hand side of \eqref{eq:conv_weights_quant_remainder_integral_2}. Indeed, otherwise take the limiting processes $\tilde{H}^0, \tilde{X}^0$ and redo Proposition \ref{prop:conv_remainder_integrals_2}. According to Lemma \ref{lem:properties_of_weighing_limit_terms}, the $\tilde{Y}^{j,0,k,\ell}$ are zero on $[0,\tilde{\tau}^{0,k,\ell}_j)$ and constant on $(\tilde{\tau}^{0,k,\ell}_j, \infty)$. Using the fact that the sets $B,C$ in the proof of Lemma \ref{lem:properties_of_weighing_limit_terms} are open and the convergence $(\tilde{Y}^{j,0,k,\ell_r}, \tilde{\tau}^{0,k,\ell_r}_j, \lceil \tilde{\tau}^{0,k,\ell_r}_j \rceil) \Rightarrow (\tilde{Y}^{j,0,k}, \tilde{\tau}^{0,k}_j , \tilde{\tau}^{0,k}_j)$ on $(\D_{\R^d}[-c,\infty), \dM)\times ([0,\infty), |\cdot|)^2$, we deduce the specific form of the single-jump process $\tilde{Y}^{j,0,k}$ as in the proof of Lemma \ref{lem:properties_of_weighing_limit_terms}.
\end{proof}

\begin{lemma} \label{lem:conv_weighing_quant_remainder_integrals_3}
    To every subsequence of $\{(H^n,X^n)\}$ and every subsequence of the natural numbers $\{\ell_r\}$ such that there is convergence \eqref{eq:conv_relevant_quant_remainder_integrals_2} and \eqref{eq:conv_weights_quant_remainder_integral_2}, there exists a subsequence $\{k_p\}$ of the natural numbers and random variables $\tilde\xi^0_{\tilde\sigma^{0,i}_j} \in [0,1]^d$---which, without loss of generality, live on the common probability space $(\tilde \Omega^0, \tilde{\mathcal{F}}^0, \tilde\Pro^0)$ from Proposition \ref{prop:conv_remainder_integrals_2} and Lemma \ref{lem:conv_weighing_quant_remainder_integrals_2}---such that 
    \begin{align}
        \Big(\tilde{H}^0, \tilde{X}^0,\big(\tilde{Y}^{j,0,k_p}_{\tilde{\sigma}^{0,i}_j}\big)_{i,j\ge 1}\Big) \; \Rightarrow \; \Big(\tilde{H}^0, \tilde{X}^0,\big(\tilde\xi^0_{\tilde\sigma^{0,i}_j}\big)_{i,j\ge 1} \Big)\label{eq:conv_final_weights_quant_remainder_integrals_3}
    \end{align}
    as $p\to \infty$ on $(\D_{\R^d}[0,\infty),\dM)^2 \times (\R^d,|\cdot|)^\N$.
\end{lemma}
\begin{proof}
    A simple application of Prokhorov's Theorem to $\{(\tilde{H}^0, \tilde{X}^0,(\tilde{Y}^{j,0,k_p}_{\tilde{\sigma}^{0,i}_j})_{i,j\ge 1})\}_{p\ge 1}$ yields the claim. Without loss of generality, we may assume the common limiting probability space to be $(\tilde \Omega^0, \tilde{\mathcal{F}}^0, \tilde\Pro^0)$ from Proposition \ref{prop:conv_remainder_integrals_2} and Lemma \ref{lem:conv_weighing_quant_remainder_integrals_2}. Indeed, otherwise take the limiting processes $\tilde{H}^0, \tilde{X}^0$ and redo Proposition \ref{prop:conv_remainder_integrals_2}.
\end{proof}

In the sequel, whenever we write "$\Rightarrow_{n,\,r,\,p\, \to \,\infty}$" it is to describe the consecutive weak limits starting with the weak limit in $n$ for each $r,p$, then the weak limit in $r$ for each $p$, and finally the weak limit in $p$.
\begin{proposition} \label{prop:pre_final_convergence_remainder_integrals}
    There exists a subsequence of $\{(H^n,X^n)\}$ (which we will denote as the original sequence), subsequences of the natural numbers $\{\ell_r\}$, $\{k_p\}$, and a probability space $(\tilde{\Omega}^0, \tilde{\mathcal{F}}^0, \tilde{\Pro}^0)$ such that  
    \begin{align*}
        \Big( H^n, \, X^n, \,\sum_{j=1}^\infty \, (H^n_{\lceil \tau^{n,k_p,\ell_r}_j\rceil}- &H^n_{\lfloor \tau^{n,k_p,\ell_r}_j\rfloor}) (X^n_{\lceil \tau^{n,k_p,\ell_r}_j\rceil}- X^n_{\lfloor \tau^{n,k_p,\ell_r}_j\rfloor}) Y^{j,n,k_p,\ell_r} \Big) \notag \\[1ex]
        &\underset{n,\,r,\,p\, \to \,\infty}{\Rightarrow} \; \; \Big( \tilde{H}^0, \, \tilde{X}^0, \sum_{ \sigma \in \mathcal{T}(\tilde{H}^0)} \, \tilde{\xi}^0_{\sigma} \Delta\tilde{H}^0_{\sigma} \Delta \tilde{X}^0_{\sigma}  \ind_{\{\sigma  \, \le \, \bullet\}}  \Big) 
        \end{align*}
        on $(\D_{\R^d}[0,\infty), \dM)^3$, where the limiting quantities all live on the common probability space $(\tilde{\Omega}^0, \tilde{\mathcal{F}}^0, \tilde{\Pro}^0)$ and $\tilde{\xi}^0_{\sigma} \in [0,1]^d$, $\sigma \in \mathcal{T}(\tilde{H}^0)$.
\end{proposition}
\begin{proof}
    Let $\{(H^n,X^n)\}$ be a further subsequence chosen according to Proposition \ref{prop:conv_remainder_integrals_2} such that the convergence \eqref{eq:conv_relevant_quant_remainder_integrals_2} holds. Then, choose a subsequence $\{\ell_r\}$ such that  \eqref{eq:conv_weights_quant_remainder_integral_2} holds, and finally a subsequence $\{k_p\}$ admitting \eqref{eq:conv_final_weights_quant_remainder_integrals_3}.
    
    Since weak convergence on $(\D_{\R^d}[0,\infty),\dM)$ is equivalent to weak convergence on $(\D_{\R^d}[0,T],\dM)$ for all $T>0$ in an unbounded subset of $[0,\infty)$, it suffices to show the weak convergence on $(\D_{\R^d}[0,T],\dM)^3$ for all $T \in \Theta_{\ell_r}$. Fix $T \in \Theta_{\ell_r}$. Then, for any $p,r \ge 1$ and on the compact intervals $[0,T]$, the number of jumps of the $H^n$ larger than $a_{k_p}$ is tight in $n$, i.e., $\operatorname{sup}_{n\ge 0} \Pro^n( \tau^{n,k_p,\ell_r}_J \le T) \to 0$ as $J\to \infty$. Therefore, it is sufficient to show 
     \begin{align}
        \Big( H^n, &\, X^n, \,\sum_{j=1}^J \, (H^n_{\lceil \tau^{n,k_p,\ell_r}_j\rceil}- H^n_{\lfloor \tau^{n,k,\ell_r}_j\rfloor}) (X^n_{\lceil \tau^{n,k_p,\ell_r}_j\rceil}- X^n_{\lfloor \tau^{n,k_p,\ell_r}_j\rfloor}) \, Y^{j,n,k_p,\ell_r} \Big) \notag \\
        &\underset{n\to \infty}{\Rightarrow} \; \; \Big( \tilde{H}^0, \, \tilde{X}^0, \,\sum_{j=1}^J \, (\tilde{H}^0_{\lceil \tilde{\tau}^{0,k_p,\ell_r}_j\rceil}- \tilde{H}^0_{\lfloor \tilde{\tau}^{0,k_p,\ell_r}_j\rfloor}) (\tilde{X}^0_{\lceil \tilde{\tau}^{0,k_p,\ell_r}_j\rceil}- \tilde{X}^0_{\lfloor \tilde{\tau}^{0,k_p,\ell_r}_j\rfloor}) \, \tilde{Y}^{j,0,k_p,\ell_r} \Big) \label{eq:first_conv_cuttoff_J_final_remainder_integrals}
        \end{align} 
        in order to obtain 
        \begin{align}
        \Big( H^n, &\, X^n, \,\sum_{j=1}^\infty \, (H^n_{\lceil \tau^{n,k_p,\ell_r}_j\rceil}- H^n_{\lfloor \tau^{n,k_p,\ell_r}_j\rfloor}) (X^n_{\lceil \tau^{n,k_p,\ell_r}_j\rceil}- X^n_{\lfloor \tau^{n,k_p,\ell_r}_j\rfloor}) \, Y^{j,n,k_p,\ell_r} \Big) \notag \\
        &\underset{n\to \infty}{\Rightarrow} \; \; \Big( \tilde{H}^0, \, \tilde{X}^0, \,\sum_{j=1}^\infty \, (\tilde{H}^0_{\lceil \tilde{\tau}^{0,k_p,\ell_r}_j\rceil}- \tilde{H}^0_{\lfloor \tilde{\tau}^{0,k_p,\ell_r}_j\rfloor}) (\tilde{X}^0_{\lceil \tilde{\tau}^{0,k_p,\ell_r}_j\rceil}- \tilde{X}^0_{\lfloor \tilde{\tau}^{0,k_p,\ell_r}_j\rfloor}) \tilde{Y}^{j,0,k_p,\ell_r} \Big) \notag
        \end{align}
        on $(\D_{\R^d}[0,T],\dM)^3$. Note that addition is a continuous operation on the $M_1$ Skorokhod space whenever the limiting objects have no common discontinuity. On account of Lemma \ref{lem:properties_of_weighing_limit_terms}, it holds that $\operatorname{Disc}(\tilde{Y}^{i,0,k_p,\ell_r}) \cap \operatorname{Disc}(\tilde{Y}^{j,0,k_p,\ell_r})=\emptyset$ almost surely whenever $i\neq j$. Hence, \eqref{eq:first_conv_cuttoff_J_final_remainder_integrals} follows from generalised continuous mapping \cite[Thm.~2.7]{billingsley} and \eqref{eq:conv_relevant_quant_remainder_integrals_2}, together with the fact that $T$ is an almost sure continuity point of all limiting quantities.
        
        Next, a similar argument as in the first part of the proof together with Lemma \ref{lem:conv_weighing_quant_remainder_integrals_2} (note in particular that the $\tilde{Y}^{j,0,k_p}=\tilde{Y}^{j,0,k_p}_{\tilde{\tau}^{0,k_p}_j} \ind_{[\tilde{\tau}^{0,k_p}_j,\, \infty)} $, $j\ge 1$, do not have common discontinuities) yields 
        \begin{align*}
        \Big( \tilde H^0, \tilde X^0, \sum_{j=1}^\infty \, (\tilde{H}^0_{\lceil \tilde{\tau}^{0,k_p,\ell_r}_j\rceil}- \tilde{H}^0_{\lfloor \tilde{\tau}^{0,k_p,\ell_r}_j\rfloor}) (\tilde{X}^0_{\lceil \tilde{\tau}^{0,k_p,\ell_r}_j\rceil}- \tilde{X}^0_{\lfloor \tilde{\tau}^{0,k_p,\ell_r}_j\rfloor}) \tilde{Y}^{j,0,k_p,\ell_r} \Big) \\
        \Rightarrow \quad \Big( \tilde H^0, \tilde X^0, \sum_{ j=1 }^\infty \, \tilde{Y}^{j,0,k_p}_{\tilde{\tau}^{0,k_p}_j} \Delta\tilde{H}^0_{\tilde{\tau}^{0,k_p}_j} \Delta \tilde{X}^0_{\tilde{\tau}^{0,k_p}_j}  \ind_{[\tilde{\tau}^{0,k_p}_j,\, \infty)}   \Big)
        \end{align*}
        on $(\D_{\R^d}[0,T],\dM)^3$ as $r\to \infty$.

        Finally, recall the definition of the stopping times $\tilde{\sigma}^{0,i}_j$ given in \eqref{eq:def_limiting_stopping_times_intervals_tightness_without_AVCI} and let us rewrite
        $$ \sum_{ j=1 }^\infty \, \tilde{Y}^{j,0,k_p}_{\tilde{\tau}^{0,k_p}_j} \Delta\tilde{H}^0_{\tilde{\tau}^{0,k_p}_j} \Delta \tilde{X}^0_{\tilde{\tau}^{0,k_p}_j}  \ind_{[\tilde{\tau}^{0,k_p}_j ,\, \infty)} \; = \; \sum_{i=1}^{k_p} \sum_{ j=1 }^\infty \, \tilde{Y}^{j,0,k_p}_{\tilde{\sigma}^{0,i}_j} \Delta\tilde{H}^0_{\tilde{\sigma}^{0,i}_j} \Delta \tilde{X}^0_{\tilde{\sigma}^{0,i}_j}  \ind_{[\tilde{\sigma}^{0,i}_j ,\, \infty)} \; =: \; \tilde{Z}^{0,k_p}.$$
        Since $(H^0,X^0)$ and $(\tilde{H}^0,\tilde{X}^0)$ coincide in law, Condition \eqref{ass:R2} and $|\tilde{Y}^{j,0,k}|\le d$ ensure that, uniformly on compacts, $\{\tilde{Z}^{0,k_p}\}$ is a Cauchy sequence with respect to convergence in probability. Thus, it suffices to investigate the limiting behaviour of 
        $$ \sum_{i=1}^{K} \sum_{ j=1 }^\infty \, \tilde{Y}^{j,0,k_p}_{\tilde{\sigma}^{0,i}_j} \Delta\tilde{H}^0_{\tilde{\sigma}^{0,i}_j} \Delta \tilde{X}^0_{\tilde{\sigma}^{0,i}_j}  \ind_{[\tilde{\sigma}^{0,i}_j ,\, \infty)} $$
        as $p\to \infty$, for each $K\ge 1$. And by a similar argument as in the first part of the proof, it even is enough to consider
        $$ \sum_{i=1}^{K} \sum_{ j=1 }^J \, \tilde{Y}^{j,0,k_p}_{\tilde{\sigma}^{0,i}_j} \Delta\tilde{H}^0_{\tilde{\sigma}^{0,i}_j} \Delta \tilde{X}^0_{\tilde{\sigma}^{0,i}_j} \ind_{[\tilde{\sigma}^{0,i}_j ,\, \infty)} $$
        as $p\to \infty$, for each $K,J\ge 1$. Applying the continuous mapping theorem and Lemma \ref{lem:conv_weighing_quant_remainder_integrals_3}, we deduce
        $$ \Big( \tilde H^0, \tilde X^0,\tilde{Z}^{0,k_p}\big) \; \Rightarrow \; \Big( \tilde H^0, \tilde X^0,\sum_{i=1}^{\infty} \sum_{ j=1 }^\infty \, \tilde\xi^0_{\tilde{\sigma}^{0,i}_j} \Delta\tilde{H}^0_{\tilde{\sigma}^{0,i}_j} \Delta \tilde{X}^0_{\tilde{\sigma}^{0,i}_j}  \ind_{[\tilde{\sigma}^{0,i}_j ,\, \infty)} \Big)$$
        as $p\to \infty$, and clearly the double sum in the third component on the right-hand side is equal to $\sum_{\sigma \in \mathcal{T}(\tilde{H}^0)} \, \tilde\xi^0_{\sigma} \Delta\tilde{H}^0_{\sigma} \Delta \tilde{X}^0_{\sigma}  \ind_{[\sigma, \, \infty)}$.
\end{proof}

\begin{proposition} \label{prop:final_weak_convergence_remainder_integrals}
There exists a subsequence of $\{(H^n,X^n)\}$ (which we will denote as the original sequence), subsequences of the natural numbers $\{\ell_r\}$, $\{k_p\}$, and a probability space $(\tilde{\Omega}^0, \tilde{\mathcal{F}}^0, \tilde{\Pro}^0)$ such that  
    \begin{align*}
        \Big( H^n, \, X^n, \,\int_0^\bullet \tilde{H}^{n,k_p,\ell_r}_{s-} \; \diff X^n_s \Big) \; \; \underset{n,\,r,\,p\, \to \,\infty}{\Rightarrow} \; \; \Big( \tilde{H}^0, \, \tilde{X}^0, \hspace{-1ex}\sum_{ \sigma \in \mathcal{T}(\tilde{H}^0)} \hspace{-1ex}\tilde{\xi}^0_{\sigma} \Delta\tilde{H}^0_{\sigma} \Delta \tilde{X}^0_{\sigma}  \ind_{[\sigma, \, \infty)}  \Big) 
        \end{align*}
        on $(\D_{\R^d}[0,\infty), \dM)^3$, where the limiting quantities all live on the common probability space $(\tilde{\Omega}^0, \tilde{\mathcal{F}}^0, \tilde{\Pro}^0)$ and $\tilde{\xi}^0_{\sigma} \in [0,1]^d$, $\sigma \in \mathcal{T}(\tilde{H}^0)$.
\end{proposition}
\begin{proof}
Based on the decomposition \eqref{decomp:remainder_integral}, combining Lemma \ref{lem:junk_vanishes_conveniently} (note that $\gamma,\eta>0$ were arbitrary and $\Pro^n(A_{n,k}^\complement)\le \eta$) and Proposition \ref{prop:pre_final_convergence_remainder_integrals} yields the claim.
\end{proof}

\subsection{Proof of our main result Theorem \ref{thm:tightness_without_AVCI}} \label{sec:proof_main_result}

We are now ready to prove Theorem \ref{thm:tightness_without_AVCI}.
\begin{proof}[Proof of Theorem \ref{thm:tightness_without_AVCI}]
    Since the topology of weak convergence on the separable metric space $(\D_{\R^{d}}[0,\infty), \dM)^3$ is metrisable, from Proposition \ref{prop:final_weak_convergence_remainder_integrals} we can extract subsequences $\{n_m\}$, $\{\ell_r\}$, $\{k_p\}$, and a probability space $(\tilde{\Omega}^0, \tilde{\mathcal{F}}^0, \tilde{\Pro}^0)$, such that for any increasing functions $\kappa_1,\kappa_2: \N \to \N$, it holds that 
    \begin{align*}
        \Big( H^{n_{\kappa_1(p)}}, \, X^{n_{\kappa_1(p)}}, \,\int_0^\bullet \tilde{H}^{n_{\kappa_1(p)},k_p,\ell_{\kappa_2(p)}}_{s-} \; \diff X^{n_{\kappa_1(p)}}_s \Big) \; \Rightarrow \; \Big( \tilde{H}^0, \, \tilde{X}^0, \hspace{-1ex}\sum_{ \sigma \in \mathcal{T}(\tilde{H}^0)} \hspace{-1.5ex} \tilde{\xi}^0_{\sigma} \Delta\tilde{H}^0_{\sigma} \Delta \tilde{X}^0_{\sigma}  \ind_{[\sigma ,\,\infty) }  \Big) 
        \end{align*}
        on $(\D_{\R^d}[0,\infty), \dM)^3$ as $p\to \infty$. Now, according to Proposition \ref{cor:convergence_without_AVCI_for_corrected_integrands}, we can find subsequences $\{n_{m(p)}\}$, $\{\ell_{r(p)}\}$ of $\{n_m\}$ and $\{\ell_r\}$ respectively with
        \begin{align*}
    		\Big( H^{n_{m(p)}}, \,X^{n_{m(p)}},\, \int_0^\bullet \, H^{n_{m(p)}}_{s-} \, &- \, \tilde{H}^{n_{m(p)},k_p,\ell_{r(p)}}_{s-} \, \diff X_s^{n_{m(p)}} \Big)\; \Rightarrow \; \Big(  \tilde H^0, \,\tilde X^0,\, \int_0^\bullet \, \tilde H^0_{s-} \, \diff \tilde X^0_s \Big) 
	   \end{align*}
    on $(\D_{\R^{d}}[0,\infty), \, \dM)^3$ as $p\to \infty$. Indeed, note that the results in Section \ref{sec:proofs_corrected_integrands} hold equally true if we replace $(H^0,X^0)$ by $(\tilde{H}^0, \tilde{X}^0)$  as both quantities coincide in law. Since $p\mapsto m(p)$, $p\mapsto r(p)$ are increasing, the convergence of the remainder integrals above holds true for these subsequences. Further, without loss of generality we may assume 
      \begin{align}
        \Big( H^{n_{m(p)}}, \, X^{n_{m(p)}}, \,&\int_0^\bullet \, H^{n_{m(p)}}_{s-} \, - \, \tilde{H}^{n_{m(p)},k_p,\ell_{r(p)}}_{s-} \, \diff X_s^{n_{m(p)}} ,\, \int_0^\bullet \tilde{H}^{n_{m(p)},k_p,\ell_{r(p)}}_{s-} \; \diff X^{n_{m(p)}}_s \Big) \notag\\[1ex]
        &\Rightarrow \; \Big( \tilde{H}^0, \, \tilde{X}^0,\, \int_0^\bullet \, \tilde H^0_{s-} \, \diff \tilde X^0_s,\, \sum_{ \sigma \in \mathcal{T}(\tilde{H}^0)} \, \tilde{\xi}^0_{\sigma} \Delta\tilde{H}^0_{\sigma} \Delta \tilde{X}^0_{\sigma}  \ind_{[\sigma ,\,\infty)}  \Big) \label{eq:joint_convergence_of_all_final_quantities_to_piece_together}
        \end{align}
        on $(\D_{\R^{d}}[0,\infty), \, \dM)^4$ as $p\to \infty$. Now, given that $(H^0,X^0)$ and $(\tilde H^0, \tilde X^0)$ coincide in law on $(\D_{\R^{d}}[0,\infty), \, \dM)^2$ endowed with its Borel $\sigma$-algebra, Condition \ref{ass:R1} also holds for the pair $(\tilde H^0, \tilde X^0)$. Therefore, it holds $\tilde \Pro^0$-almost surely that coordinatewise
        $$ \Delta \Big( \int_0^t \, \tilde H^0_{s-} \, \diff \tilde X^0_s\Big) \Delta \Big( \sum_{ \sigma \in \mathcal{T}(\tilde{H}^0)} \, \tilde{\xi}^0_{\sigma} \Delta\tilde{H}^0_{\sigma} \Delta \tilde{X}^0_{\sigma}  \ind_{\{\sigma  \, \le \, t\}}\Big) \; \ge \; 0$$
        for all $t\ge 0$, since $\operatorname{Disc}_{\tilde\Pro^0}(\tilde X^0)$ is at most countable. Indeed, at any $t\ge 0$, the product of the respective jumps is given coordinatewise by
        $$\tilde H^0_{t-}\Delta\tilde X^0_t \tilde\xi^0_t\Delta\tilde H^0_t\Delta\tilde X^0_t \; = \; \tilde\xi^0_t\tilde H^0_{t-}\Delta\tilde H^0_t \lvert\Delta\tilde X^0_t\rvert^2 \; \geq \; 0,$$
where the inequality follows from Condition \eqref{ass:R1} and $0\leq\widetilde\xi^0_t\leq1$.
        Making use of the fact that the map $(\alpha,\beta)\mapsto \alpha+\beta$ is continuous in $(\D_{\R^d}[0,\infty),\dM)^2$ at all points $(\alpha,\beta)$ such that $\Delta \alpha^{(i)}(t) \Delta \beta^{(i)}(t)\ge 0$ for all components $1\le i\le r$ and all $t\ge 0$ (see e.g. \cite[Thm.~12.7.3]{whitt}), the generalised continuous mapping theorem yields 
        \begin{align}
         \Big( &H^{n_{m(p)}}, \, X^{n_{m(p)}}, \,\int_0^\bullet \, H^{n_{m(p)}}_{s-} \, \diff X_s^{n_{m(p)}} \Big) \notag \\[1ex]
        &=  \Big( H^{n_{m(p)}}, \, X^{n_{m(p)}}, \,\int_0^\bullet \, H^{n_{m(p)}}_{s-}  -  \tilde{H}^{n_{m(p)},k_p,\ell_{r(p)}}_{s-} \, \diff X_s^{n_{m(p)}} \, +\, \int_0^\bullet \tilde{H}^{n_{m(p)},k_p,\ell_{r(p)}}_{s-} \; \diff X^{n_{m(p)}}_s \Big) \notag\\[1ex]
        &\Rightarrow \; \Big( \tilde{H}^0, \, \tilde{X}^0,\, \int_0^\bullet \, \tilde H^0_{s-} \, \diff \tilde X^0_s\, + \, \sum_{ \sigma \in \mathcal{T}(\tilde{H}^0)} \, \tilde{\xi}^0_{\sigma} \Delta\tilde{H}^0_{\sigma} \Delta \tilde{X}^0_{\sigma}  \ind_{[\sigma ,\,\infty)}  \Big) \label{eq:very_final_convergence}
        \end{align}
        on $(\D_{\R^{d}}[0,\infty), \, \dM)^3$ as $p\to \infty$. The semimartingale property of $\tilde X^0$ with respect to the extended filtration follows directly from \cite[Prop.~2.5]{andreasfabrice_theorypaper} by virtue of \eqref{eq:joint_convergence_of_all_final_quantities_to_piece_together} and the fact that the $X^n$ satisfy \eqref{eq:Mn_An_condition}.
\end{proof}

\subsection{Proof of Proposition \ref{cor:integral_conv_without_AVCI} } \label{sec:proof_proposition_anti_AVCI}
In this section, we will give a brief outline of how to prove Proposition \ref{cor:integral_conv_without_AVCI}. We will proceed exactly as in the proofs in Sections \ref{sec:proofs_corrected_integrands}--\ref{sec:proof_main_result}, except for the limiting behaviour of the first summand on the right-hand side of \eqref{decomp:remainder_integral}. Under the assumptions of Proposition \ref{cor:integral_conv_without_AVCI}, the integral part of this term can be shown to asymptotically equal one. For simplicity, let us assume $d=1$. The proof for $d>1$ can be adapted in a straightforward fashion.

Carry out the proof in full analogy with that of Theorem \ref{thm:tightness_without_AVCI} up to \eqref{decomp:remainder_integral}, where for the $\xi^{n,k,\ell}_j$ in the decomposition \eqref{decomp:X^n_for_remainder_integrals} we make use of the specific construction \eqref{eq:interpolation_2_monotone_piece}, while the $\zeta^{n,k,\ell}_j$ in decomposition \eqref{decomp:H^n_for_remainder_integrals_non_adapted} are constructed according to \eqref{eq:interpolation_1_monotone_piece}. Assume that the elements of $T\cup \Theta_\ell$ are continuity points of all $H^n$, $n\ge 0$. Indeed, this comes with no loss of generality as the sets $\operatorname{Disc}_{\Pro^n}(H^n)$ are known to be countable, and therefore this does not disturb the construction of the sets $\Theta_\ell$ at the beginning of Section \ref{sec:proofs}. Further, let $\ell\ge 1$ be large enough so that, due to \eqref{eq:cond_integral_conv_without_AVCI}, we have that
 \begin{align*} \Pro^n \Big( \sup{}\{|X^{n}_u-X^{n}_s|\wedge |H^{n,k,\ell}_{s-}- H^{n,k,\ell}_{r-}| \, : \,\lfloor \tau^{n,k,\ell}_j \rfloor \le u< s<r\le \rho^{n,k,\ell}_j\wedge T, \, j\ge 1 \}>\frac{\gamma}{4}\Big)  
    \end{align*}
is bounded by $ \eta$ for all $n,k\ge 1$. Indeed, this is possible since, by definition, $\operatorname{sup}_{j\ge 1}|\rho^{n,k,\ell}_j-\tau^{n,k,\ell}_j|\le |\Theta_\ell| \to 0$ as $\ell \to \infty$. Moreover, on the event 
$$\Gamma_{n,k,\ell}(\gamma) :=  \Big\{\sup{}\{|X^n_u-X^n_s|\wedge |H^n_{s-}-H^n_{r-}| \, : \, \lfloor \tau^{n,k,\ell}_j \rfloor \le u< s<r\le \rho^{n,k,\ell}_j\wedge T, \, j\ge 1 \}\, \le \, \frac \gamma 4\Big\}$$
we have, in particular, that $|X^n_{\lfloor \tau^{n,k,\ell}_j \rfloor} - X^n_s| > \gamma$ implies $|H^n_{s-}-H^n_{(\rho^{n,k,\ell}_j \wedge T)-}|\le \gamma/4$ for all $j\ge 1$ and $\lfloor \tau_j^{n,k,\ell}\rfloor \le s\le \rho_j^{n,k,\ell}\wedge T$. Inspecting construction \eqref{eq:interpolation_2_monotone_piece} of the 
non-decreasing step function $\xi$ in the proof of Lemma \ref{lem:approx_of_M1_by monotone_pieces_X}, we deduce that the following holds true for the quantities $\xi^{j,n,k,\ell}$ in \eqref{decomp:X^n_for_remainder_integrals} on $\Gamma_{n,k,\ell}(\gamma) \cap A_{n,k}$: if we define 
\begin{align*}
    \sigma^{n,k,\ell}_j&:=\operatorname{inf}\big \{s>\lfloor \tau^{n,k,\ell}_j \rfloor \, : \, |X^n_{\lfloor \tau^{n,k,\ell}_j \rfloor } - X^n_s| \, > \, \gamma \big\} \wedge \rho^{n,k,\ell}_j \wedge T,
\end{align*}
then $\xi^{j,n,k,\ell}\equiv 0$ on $[\lfloor \tau^{n,k,\ell}_j\rfloor, \sigma^{n,k,\ell}_j)$. Clearly, it must also hold that 
$$ \operatorname{inf}\big \{s>\lfloor \tau^{n,k,\ell}_j \rfloor \, : \, |H^n_{(\rho^{n,k,\ell}_j\wedge T)-} - H^n_{s-}| \, \le \, \gamma/4 \big\} \; \le \;\sigma^{n,k,\ell}_j $$
since we have assumed to be on the event $\Gamma_{n,k,\ell}(\gamma)$. Thus, by construction \eqref{eq:interpolation_1_monotone_piece} of the processes $\zeta^{j,n,k,\ell}$ and since $\rho^{n,k,\ell}_j$ and $T$ were chosen to be continuity points of $H^n$, according to Lemma \ref{lem:approx_of_M1_by monotone_pieces_X} we obtain that $\zeta^{j,n,k,\ell}\equiv 1$ on $( \bar{\sigma}^{n,k,\ell}_j , \rho^{n,k,\ell}_j\wedge T]$ for some $\bar{\sigma}^{n,k,\ell}_j<\sigma^{n,k,\ell}_j$ (where $|H^n_{\bar{\sigma}^{n,k,\ell}_j}-H^n_{\rho^{n,k,\ell}_j\wedge T}|\le \gamma/2$). Thus, from $\xi^{j,n,k,\ell}\equiv 0$ on $[\lfloor \tau^{n,k,\ell}_j\rfloor, \sigma^{n,k,\ell}_j)\supset [\lfloor \tau^{n,k,\ell}_j\rfloor, \tilde{\sigma}^{n,k,\ell}_j]$, \, $\xi^{j,n,k,\ell}_{\rho^{n,k,\ell}_j\wedge T}=1$, and $\zeta^{j,n,k,\ell}\equiv 1$ on $(\tilde{\sigma}^{n,k,\ell}_j , \rho^{n,k,\ell}_j\wedge T]$ we deduce
$$Y^{j,n,k,\ell}_{\rho^{n,k,\ell}_j\wedge T} \; = \;\int_{\tau^{n,k,\ell}_j\wedge T }^{\rho^{n,k,\ell}_j\wedge T} \ind_{(\tilde{\sigma}^{n,k,\ell}_j, \, \rho^{n,k,\ell}_j\wedge T]} \, \diff \xi^{j,n,k,\ell}_s \; = \; \xi^{j,n,k,\ell}_{\rho^{n,k,\ell}_j\wedge T} - \xi^{j,n,k,\ell}_{\tilde{\sigma}^{n,k,\ell}_j} \; = \; 1-0 \; =\;  1.$$

With this information at hand, first replace all instances of $A_{n,k}$ by $\Gamma_{n,k,\ell}(\gamma) \cap A_{n,k}$ and continue with the proofs from Lemma \ref{lem:junk_vanishes_conveniently} to Lemma \ref{lem:conv_weighing_quant_remainder_integrals_2}. In analogy to the proof of Lemma \ref{lem:properties_of_weighing_limit_terms}, we can then establish that $\tilde{Y}^{j,0,k,\ell}\equiv 1$ on $[\lceil \tilde{\tau}^{0,k,\ell}_j\rceil, \infty)$ and, by definition, $\tilde{Y}^{j,0,k,\ell}\equiv 0$ on $[0,\lfloor \tilde{\tau}^{0,k,\ell}_j\rfloor)$. And finally, as in the proof of Lemma \ref{lem:conv_weighing_quant_remainder_integrals_2} this translates to the $\tilde{Y}^{j,0,k}$ so that they are of the form 
$$\tilde{Y}^{j,0,k} \; = \; \ind_{[\tilde{\tau}^{0,k}_j, \, \infty)}.$$
Now follow the procedure until the end of Section \ref{sec:proof_main_result} and note that, due to the above, $\tilde{\xi}^0_\sigma\equiv 1$ for all $\sigma \in \mathcal{T}(\tilde{H}^0)$. This, however, implies that the weak limit in \eqref{eq:tightness_int_conv} is 
\begin{align*}
    \mathcal{L}\Big( \tilde H^0, \, &\tilde{X}^0,\,\int_0^\bullet \, \tilde H^0_{s-} \; \diff \tilde X^0_s \; + \sum_{\sigma \in \mathcal{T}(\tilde H^0)} \hspace{-1ex}   \Delta \tilde H^0_{\sigma} \, \Delta \tilde X^0_{\sigma} \, \ind_{[\sigma ,\,\infty)}\Big) \; = \\
     &\mathcal{L}\Big(  H^0, \, X^0,\,\int_0^\bullet \,  H^0_{s-} \; \diff X^0_s \; + \sum_{\sigma \in \mathcal{T}(H^0)} \hspace{-1ex}  \Delta H^0_{\sigma} \, \Delta X^0_{\sigma} \, \ind_{[\sigma ,\,\infty) }\Big)
\end{align*} 
which does not depend on the specific subsequence, hence there is actual weak convergence \eqref{eq:int_conv_without_AVCI}.

\subsection{Proof of Theorem \ref{cor:main_result_M1_J1_case} } \label{sec:proof_corollary_J1_integrators_main_result}
\begin{proof}[Proof of Theorem \ref{cor:main_result_M1_J1_case}]
    The proof follows the exact same structure of the proof of Theorem \ref{thm:tightness_without_AVCI} in Sections \ref{sec:proofs_corrected_integrands}--\ref{sec:proof_main_result}, and we will simply highlight the instances where specific modifications are required.  

    Define sets
    $$ \tilde A_{n,\ell}  := \biggl\{ \inf{} \biggl\{ \sup{i\le r} \sup{s,t \, \in \, [t_{i-1},t_{i})} |X^n_s-X^n_t|\; : \; t_0<...<t_r \text{ and } \inf{1\le i \le r}(t_i-t_{i-1})\ge |\Theta_\ell|\biggr\} \,> \, \frac \gamma 2 \biggr\}$$
    with $t_0=0$ and $t_r=T$, where the outer infimum runs over all partitions of $[0,T]$ with mesh size no smaller than $|\Theta_\ell|$. Due to the $J_1$ tightness of the $X^n$, by \cite[Thm.~13.2]{billingsley} it holds that $\limsup_{\ell\to \infty}\limsup_{n\to \infty} \Pro^n(\tilde A_{n,k,\ell})=0$. Therefore, note that replacing $A_{n,k}$ in \eqref{eq:auxiliary_set_probability_bound_proof_remainder_integrals} by $A_{n,k,\ell}:=A_{n,k}\cap \tilde{A}_{n,\ell}$ leaves the proof of Theorem \ref{thm:tightness_without_AVCI} unaffected. With regard to Section \ref{sec:proof_dealing_with_the_remainder_integrals}, the following is clearly true on $\tilde{A}_{n,\ell}^\complement$: for all $k,j\ge 1$ there exists a random time $\kappa^{n,k,\ell}_j \in [\tau^{n,k,\ell}_j\wedge T,\rho^{n,k,\ell}\wedge T)$ such that 
    $$ \operatorname{max} \biggl\{ \sup{\tau^{n,k,\ell}\le r,s < \kappa^{n,k,\ell}_j}|X^n_r-X^n_s| \; , \; \sup{\kappa^{n,k,\ell}_j\le r,s < \rho^{n,k,\ell}}|X^n_r-X^n_s| \biggr\}\; \le \; \frac \gamma 2.$$
    Consequently, the $\xi^{j,n,k,\ell}$ from \eqref{decomp:X^n_for_remainder_integrals}---obtained by virtue of Lemma \ref{lem:approx_of_M1_by monotone_pieces_X} with choice \eqref{eq:interpolation_1_monotone_piece}---are single-jump processes of the form $\xi^{j,n,k,\ell} = \ind_{[\kappa^{n,k,\ell}_j ,\, \infty)}$, and the $Y^{j,n,k,\ell}$ in \eqref{eq:defi_scaling_terms_Y_n} become 
    $$Y^{j,n,k,\ell}=\zeta^{j,n,k,\ell}_{\kappa^{n,k,\ell}_j-} \ind_{[\kappa^{n,k,\ell}_j, \, \infty)}$$
    where we recall $0\le \zeta^{j,n,k,\ell}\le 1$. Evidently, these processes are tight in the $J_1$ topology, and so all convergences related to the $Y^{j,n,k,\ell}$ and $\tilde Y^{j,0,k,\ell}$ in Section \ref{sec:proof_dealing_with_the_remainder_integrals} can be seen to hold in $J_1$ instead of $M_1$. Thus, it is straightforward to show that the convergences in Proposition \ref{prop:pre_final_convergence_remainder_integrals} and Proposition \ref{prop:final_weak_convergence_remainder_integrals}, in fact, hold on $(\D_{\R^d}[0,\infty), \dM) \times (\D_{\R^{d}}[0,\infty), \dJ)^2$. As a consequence, also the convergence in \eqref{eq:joint_convergence_of_all_final_quantities_to_piece_together} can be strengthened, namely to be on $(\D_{\R^d}[0,\infty), \dM) \times (\D_{\R^{d}}[0,\infty), \dJ)^3$. Finally, since $\big|\Delta \int_0^s  H^{n}_{r-}  - \tilde{H}^{n,k,\ell}_{r-}  \diff X_r^{n}\big| \vee  \big|\Delta \int_0^s \tilde{H}^{n,k,\ell}_{r-}  \diff X_r^{n}\big| \le 2|H^n|_s^* |\Delta X^n_s|$ for all $s>0$ almost surely, we can apply Lemma \ref{lem:joint_convergence_from_coinciding_jump_times} together with the continuous mapping theorem and \cite[Prop.~A.3]{andreasfabrice_theorypaper} in order to obtain \eqref{eq:very_final_convergence} on $(\D_{\R^d}[0,\infty), \dM) \times (\D_{\R^{2d}}[0,\infty), \dJ)$.
\end{proof}

\setcounter{section}{0}
\renewcommand\thesection{\Alph{section}}
\section{Appendix}

In this appendix we collect a few technical auxiliary results and we give the proofs of Lemmas \ref{lem:approx_of_M1_by monotone_pieces_X} \& \ref{lem:approx_of_M1_by monotone_pieces_H}, and \ref{lem:conv_relevant_quant_remainder_integrals}.\\[2ex]
\textbf{Proofs of the short-interval monotone approximation Lemmas \ref{lem:approx_of_M1_by monotone_pieces_X} \& \ref{lem:approx_of_M1_by monotone_pieces_H}}
\begin{proof}[Proof of Lemma \ref{lem:approx_of_M1_by monotone_pieces_X}]
    Define $\nu(\lambda):=x(t_1)+\lambda(x(t_2)-x(t_1))$, set $\sigma_0:=t_1$, $\xi(\sigma_0):=0$ as well as iteratively
    \begin{align*}
        \sigma_j \; &:= \; \operatorname{inf}\{ t> \sigma_{j-1}\, : \, |x(t)-\nu(\xi(\sigma_{j-1}))| > \gamma\} \, \wedge \, t_2\\
        \xi(\sigma_j) \; &:= \; \operatorname{arg\,min}_{\lambda \in [0,1]} |x(\sigma_j)-\nu(\lambda)|
    \end{align*}
    for $j\ge 1$. Further, define $\bar{\sigma}:= \operatorname{inf}\{ t> 0 : |x(t)-x(t_2)| \le \gamma/2\} \wedge t_2$. Since $t_2-t_1<\delta$ and $w'(x,\delta)\le \gamma/2$, it holds that $|x(\sigma_j)-\nu(\xi(\sigma_j))|\le \gamma/2$, $j\ge 1$. By definition of the $\sigma_j$ and a simple application of the triangle inequality, we thus obtain $|x(\sigma_j)-x(\sigma_{j+1})|\ge \gamma/2$ for all $j\ge 1$. Set either 
    \begin{align}
        \xi(t) \; := \; \sum_{j=0}^\infty \xi(\sigma_j) \ind_{[\sigma_j\wedge \bar{\sigma},\, \sigma_{j+1}\wedge \bar{\sigma})}(t) \, + \, \ind_{[\bar\sigma, \, t_2]}(t), \qquad t_1\le t\le t_2, \label{eq:interpolation_1_monotone_piece}
    \end{align}
    or 
    \begin{align}
        \xi(t) \; := \; \sum_{j=0}^\infty \xi(\sigma_j) \ind_{[\sigma_j\wedge t_2,\, \sigma_{j+1}\wedge t_2)}(t) \, + \, \ind_{\{t_2\}}(t), \qquad t_1\le t\le t_2. \label{eq:interpolation_2_monotone_piece}
    \end{align}
    In both cases, the sum is---in fact---finite. Indeed, on compact time intervals  the number of increments of a càdlàg function larger than $\gamma/2$ is bounded and, as described above, $|x(\sigma_j)-x(\sigma_{j+1})|\ge \gamma/2$ for all $j\ge 1$. Moreover, from the definition of the quantities it clearly holds $|x(t)-\nu(\xi(t))|\le \gamma$ (in the case of \eqref{eq:interpolation_2_monotone_piece}, on $[\bar \sigma, t_2]$ this holds due to $|x(\bar \sigma)- x(t_2)|\le \gamma/2$ and $w'(x,\delta)\le \gamma/2$), and so it only remains to show that $\xi$ is monotone.\\
    Let us first assume $x(\sigma_1)> \nu(\xi(\sigma_0))=x(t_1)$. Then, we immediately deduce $x(t_2)>x(t_1)$. Indeed, $x(t_2)>x(t_1)$ must hold since, otherwise, $x(\sigma_1)\ge x(t_1)+\gamma > x(t_1)\ge x(t_2)$, and consequently there is ~$\operatorname{inf}\{|x(\sigma_1)-y|: y \in [x(t_1)\wedge x(t_2),x(t_1)\vee x(t_2)]\}\ge \gamma$, which is a contradiction to $w'(x,\delta)\le \gamma/2$. Similarly, we obtain that $x(\sigma_j)\ge \nu(\xi(\sigma_j))$ for all $j\ge 1$. To see this, assume for the sake of a contradiction, that $j$ is the first index with $x(\sigma_j)< \nu(\xi(\sigma_j))$, which yields $x(\sigma_j)\ge  x(\sigma_{j-1})\ge \nu(\xi(\sigma_{j-1}))$. Then, $\nu(\xi(\sigma_{j-1}))\le x(\sigma_{j-1})\le x(\sigma_j) < \nu(\xi(\sigma_j)) \le x(t_2)$, which yields that $x(\sigma_j) \in [\nu(\xi(\sigma_{j-1})), x(t_2)]$. Thus, there exists $\lambda$ such that $x(\sigma_j)=\nu(\lambda)$, implying, by definition of $\xi$, that $\xi(\sigma_j)=\lambda$, which contradicts $x(\sigma_j) < \nu(\xi(\sigma_j))$. Hence, it must hold that $x(\sigma_j) \ge x(\sigma_{j-1})\ge \nu(\xi(\sigma_{j-1}))$ for all $j\ge 1$, and therefore $x(\sigma_j) \in [\nu(\xi(\sigma_{j-1})),x(t_2)]$. Finally, this immediately gives us that $x(\sigma_j)=\nu(\alpha)$ with $\alpha\ge \xi(\sigma_{j-1})$, implying $\xi(\sigma_j)=\alpha\ge \xi(\sigma_{j-1})$ and hence the monotonicity of $\xi$. Monotonicity in the case $x(\sigma_1)\le x(t_1)$ can be shown analogously.
\end{proof}

\begin{proof}[Proof of Lemma \ref{lem:approx_of_M1_by monotone_pieces_H}]
Set $\sigma_0:=t_1$ and define iteratively
    \begin{align*}
        \sigma_j \; &:= \; \operatorname{inf}\{ t> \sigma_{j-1}\, : \, |x(t)-x(\sigma_{j-1})| > \gamma\} \, \wedge \, t_2
    \end{align*}
    for $j\ge 1$, as well as $\xi$ as in \eqref{eq:lem:approx_of_M1_by monotone_pieces_H}. First consider the case that $x(t_1) \le x(t_2)$. Due to $w'(x,\delta)<\gamma/2$, it is immediate that $x(t) \in (x(t_1)-\gamma/2, x(t_2)+\gamma/2)$ for all $t_1\le t\le t_2$. Assume for the sake of a contradiction that $j\ge 0$ is the first index such that $x(\sigma_{j+1})<x(\sigma_j)$. By definition of the stopping times this implies $x(\sigma_{j+1})\le x(\sigma_j)-\gamma$ and $x(\sigma_j)< x(t_2)+\gamma/2$, and so we obtain ~$\operatorname{inf}\{|x(\sigma_{j+1})-y|: y \in [x(\sigma_j)\wedge x(t_2),x(\sigma_j)\vee x(t_2)]\}\ge \gamma$. This, however, contradicts $w'(x,\delta)< \gamma/2$, since $t_1< \sigma_j <\sigma_{j+1}\le t_2\le t_1+\delta$. Thus, $\xi$ is increasing on $[t_1,t_2]$. For the case $x(t_1) > x(t_2)$ one can proceed similarly, deducing that $\xi$ is decreasing.
\end{proof}

\noindent \textbf{Towards a proof of Lemma \ref{lem:conv_relevant_quant_remainder_integrals}}
\begin{lemma} \label{lem:existence_of_suitable_sequence_not_hit_by_jumps}
    The set 
    $$W \; := \; \{ a>0 \, : \, \exists s>0 \, \text{ with }\,  \Pro^0(|\Delta H^0_s|=a \text{ or } |\Delta X^0_s|=a)>0\}$$
    is at most countable.
\end{lemma}
\begin{proof}
Assume for the sake of a contradiction that there exists $\varepsilon>0$, an interval $(b_1,b_2) \subset [0,\infty)$, and uncountably many $a \in (b_1,b_2) $ such that $\Pro^0(|\Delta H^0_{\phi(a)}|=a \text{ or } |\Delta X^0_{\phi(a)}|=a)>\varepsilon$ for some map $\phi:(b_1,b_2)\to [0,\infty)$. If the set 
    $$ \Lambda \; := \; \{ \phi(a) \, : \, \Pro^0(|\Delta H^0_{\phi(a)}|=a \text{ or } |\Delta X^0_{\phi(a)}|=a)\, > \, \varepsilon\}$$ 
    is countable, then there is at least one $s \in \Lambda$ such that $\phi$ maps infinitely many $a\in (b_1,b_2)$ to $s$. By $\sigma$-additivity of $\Pro^0$, this yields a contradiction to the finiteness of the probability measure $\Pro^0$. Hence, $\Lambda$ must be uncountable and therefore, in particular, there is an interval $[0,T]$ containing infinitely many distinct $\phi(a) \in \Lambda$, which implies
    $$ \Pro^0(|\Delta H^0_{\phi(a)}|>b_1 \text{ or } |\Delta X^0_{\phi(a)}|>b_1)\, > \, \varepsilon$$
    for infinitely many distinct $\phi(a)$, since $b_1<a$. This however, yields a contradiction since $\{s>0: \Pro^0(|\Delta H^0_s| \vee |\Delta X^0_s| > c) > \beta)\cap [0,T]$ is known to be finite for each $c,\beta>0$.
\end{proof}

For $t\ge 0$, $\mu>0$ and $a>0$ we define
\begin{align}
    \varsigma_{a,t,\mu}(\alpha) \; := \; \operatorname{inf}\{s>t \, : \, \operatorname{sup} \{ |\alpha^{(i)}(r)-\alpha^{(i)}(s)| \, : \, t \vee (
    s- \mu) \le r \le s, \, 1\le i \le d\} > a \}. \label{eq:defi_stopping_times_for_appendix_lemma_on_defi_of_suitable_sequence}
\end{align}
for all $\alpha \in \D_{\R^d}[0,\infty)$, where $\alpha^{(i)}$ denotes the $i$-th coordinate of $\alpha$. Further let $\operatorname{Disc}(\alpha)=\{s>0: |\Delta \alpha(s)|>0\}$. Following the style of \cite[pp.~340-341]{shiryaev}, we state the following lemma.

\begin{lemma} \label{lem:prerequisite_suitable_sequence_of_large_increments_stopping_time_defi}
    Under the notation of \eqref{eq:defi_stopping_times_for_appendix_lemma_on_defi_of_suitable_sequence} it holds that
    \begin{enumerate}[(i)]
        \item for all $\alpha \in \D_{\R^d}[0,\infty)$, $t\ge 0$ and $\mu>0$, the map $a \mapsto \varsigma_{a,t,\mu}(\alpha)$ is non-decreasing and right-continuous,
        \item for given $t\ge 0$, $\mu>0$ and $\alpha \in \D_{\R^d}[0,\infty)$, the set 
        \begin{align*}
            V_{t,\mu}(\alpha) \; := \; \{ a>0 \; : \; \varsigma_{a^-,t,\mu}(\alpha) < \varsigma_{a,t,\mu}(\alpha) \} ,
        \end{align*} 
        where $\varsigma_{a^-,t,\mu}(\alpha):= \lim_{\tilde{a}\uparrow a} \varsigma_{\, \tilde{a},t,\mu}(\alpha)$, is at most countable,
        \item for all $t\ge 0$, $\mu>0$ and $a>0$, the maps $\alpha \mapsto \varsigma_{a,t,\mu}(\alpha)$ and $\alpha \mapsto \varsigma_{a^-,t,\mu}(\alpha)$ are Borel-measurable on $(\D_{\R^d}[0,\infty), \dM)$,
        \item for all $t\ge 0$, $\mu>0$ and $a>0$, the map $\alpha \mapsto \varsigma_{a,t,\mu}(\alpha)$, $(\D_{\R^d}[0,\infty),\dM) \to ([0,\infty),|\cdot|)$ is continuous at each point $\alpha$ satisfying $a \notin V_{t,\mu}(\alpha)$, $t\notin \operatorname{Disc}(\alpha)$ and $\mu \notin \{|x-y|: x,y \in \operatorname{Disc}(\alpha)\}$.
    \end{enumerate}
\end{lemma}
\begin{proof}
    (i) is trivial. (ii) As a non-decreasing, right-continuous map, $a \mapsto \varsigma_{a,t,\mu}(\alpha)$ has at most countably many discontinuities, hence $V_{t,\mu}(\alpha)$ is at most countable. (iii) follows the standard measurability argument for first-passage-type functionals on Skorokhod space: by right-continuity, the defining condition may be restricted to rational time points, after which measurability follows from the fact that the Borel $\sigma$-algebra on $\D_{\R^d}[0,\infty)$ is generated by the coordinate projections (see, e.g., \cite[Sec.~11.5.3]{whitt}, and for related first-passage mappings \cite[Thm.~7.2]{whitt2}). For brevity of this account, we do not provide a detailed proof.
    (iv) Let $\dM(\alpha_n , \alpha)\to 0$ and $a \notin V_{t,\mu}(\alpha)$. We begin by showing that $\limsup_{n\to \infty} \varsigma_{a,t,\mu}(\alpha_n)\le \varsigma_{a,t,\mu}(\alpha)$. Let $\varepsilon>0$ and choose $r,s \in [0,\infty) \setminus \operatorname{Disc}(\alpha)$ such that $t\vee (s-\mu)\le r \le s$ as well as $s<\varsigma_{a,t,\mu}(\alpha)+\varepsilon$, and $|\alpha(r)-\alpha(s)|>a$. M1 convergence implies pointwise convergence at continuity points of the limit (see \cite[Lem.~12.5.1]{whitt}) and so we deduce $|\alpha_n(r)-\alpha_n(s)|>a$ for large enough $n$. Hence, $\varsigma_{a,t,\mu}(\alpha_n)\le s < \varsigma_{a,t,\mu}(\alpha)+\varepsilon$ and, since $\varepsilon$ was arbitrary, it follows $\limsup_{n\to \infty} \varsigma_{a,t,\mu}(\alpha_n)\le \varsigma_{a,t,\mu}(\alpha)$. It remains to show $\liminf_{n\to \infty} \varsigma_{a,t,\mu}(\alpha_n)\ge \varsigma_{a,t,\mu}(\alpha)$. Fix a continuity point $T\ge \varsigma_{a,t,\mu}(\alpha)$. Without loss of generality, assume towards a contradiction that there are $\varepsilon>0$, $1\le i \le d$, and times $r_n$ with $t\vee (\varsigma_{a,t,\mu}(\alpha_n)-\mu)\le r_n \le \varsigma_{a,t,\mu}(\alpha_n)<\varsigma_{a,t,\mu}(\alpha)-\varepsilon$ such that $|\alpha^{(i)}_n(r_n)-\alpha_n^{(i)}(\varsigma_{a,t,\mu}(\alpha_n))|>a$ for all $n\ge 1$. Clearly, we may therefore also assume that $d=1$, i.e. $\alpha^{(i)}=\alpha$.
    Now, let $\Pi_n$ be the set of all parametric representations $(u_n,p_n)$ of $\alpha_n$, $n\ge 1$, on $[0,T]$; and analogously define $\Pi$ for $\alpha$. Recall that $\dM(\alpha_n,\alpha)\to 0$ as $n\to \infty$ implies 
        \begin{align*} \operatorname{inf}\{ |u_n-u|^*_1 \vee |p_n-p|^*_1 \, : \, (u_n,p_n) \in \Pi_n, (u,p) \in \Pi\} \; \to \; 0 
        \end{align*}
    as $n\to \infty$. For an introduction to parametric representations and path convergence in $M_1$, see \cite[pp.~80-83]{whitt}. Choose parametric representations $(u_n,p_n), (u,p)$ with $|u_n-u|^*_1 \vee |p_n-p|^*_1 \to 0$. Clearly, there exist $z_n,\tilde{z}_n \in [0,1]$ such that $p_n(z_n)=\varsigma_{a,t,\mu}(\alpha_n), p_n(\tilde{z}_n)=r_n$, as well as $|u_n(z_n)-u_n(\tilde{z}_n)|>a$. Without loss of generality, assume $z_n\to z$, $\tilde{z}_n\to \tilde{z}$, for some $z,\tilde{z} \in [0,1]$. Otherwise pass to a subsequence. An application of the triangle inequality yields
    \begin{align*}
        a \; < \; |u_n(z_n) - u_n(\tilde{z}_n)| \; \le \; 2|u_n-u|^*_1 \; + \; |u(z_n) - u(\tilde{z}_n)|.
    \end{align*}
    By continuity of $u$, we deduce $|u(z)-u(\tilde{z})|\ge a$. Moreover, another simple application of the triangle inequality and the continuity of $p$ imply that $\varsigma_{a,t,\mu}(\alpha_n)=p_n(z_n)\to p(z)$ and $p_n(\tilde z_n)\to p(\tilde z)$, yielding further that $t\vee (p(z)-\mu)\le  p(\tilde{z}) \le  p(z) \le \varsigma_{a,t,\mu}(\alpha)-\varepsilon$. Let $\delta>0$ and distinguish four cases: 
    (1) if $p(z),p(\tilde{z}) \notin \operatorname{Disc}(\alpha)$, then it holds $u(z)=\alpha(p(z))$, $u(\tilde{z})=\alpha(p(\tilde{z}))$, and thus $a \le |u(z)-u(\tilde{z})| =|\alpha(p(z))-\alpha(p(\tilde z))|$ implies $\varsigma_{a^-,t,\mu}(\alpha) \le \varsigma_{a,t,\mu}(\alpha)-\varepsilon$. Since $a \notin V_{t,\mu}(\alpha)$ we know $\varsigma_{a^-,t,\mu}(\alpha)=\varsigma_{a,t,\mu}(\alpha)$, hence there is a contradiction. 
    (2) if $p(z) \in \operatorname{Disc}(\alpha)$ and $p(\tilde{z})=p(z)$, then due to the fact that $u(z), u(\tilde{z})$ lie on the completed graph of $\alpha$, $u(z),u(\tilde{z}) \in [\alpha(p( z)-) \wedge \alpha(p( z)), \alpha(p( z)-) \vee \alpha(p(z))]$. From $|u(z)-u(\tilde{z})|\ge a$ it thus follows $|\alpha(p( z))-\alpha(p( z)-)|\ge a$. Hence, there exists a time $r$ with $t\vee (p(z)-\mu)<r<p(z)\le \varsigma_{a,t,\mu}(\alpha)-\varepsilon$ such that $|\alpha(r)-\alpha(p(z))|> a-\delta$, and as a consequence $\varsigma_{a-\delta,t,\mu}(\alpha)\le \varsigma_{a,t,\mu}(\alpha)-\varepsilon$. Since $\delta$ was arbitrary, we deduce $\varsigma_{a^-,t,\mu}(\alpha) \le \varsigma_{a,t,\mu}(\alpha)-\varepsilon$, which yields a contradiction as in the first case.
    (3) if $p(z) \in \operatorname{Disc}(\alpha)$ and $p(\tilde z) < p(z)$, then due to $t\notin \operatorname{Disc}(\alpha)$ and $\mu \notin \{|x-y|: x,y \in \operatorname{Disc}(\alpha)\}$ we deduce $t\vee (p(z)-\mu) \notin \operatorname{Disc}(\alpha)$. Since $u(z)$ lies on the completed graph of $\alpha$, as in case (2) from $|u(z)-u(\tilde{z})|\ge a$ we deduce the existence of a time $s$ with $t\vee(p(z)-\mu)\le p(\tilde{z})< s\le p(z)$ and $|\alpha(s)-u(\tilde{z})|> a-\delta$.  Also $u(\tilde{z})$ lies on the completed graph and so there is a time $r$ with $t\vee(p(z)-\mu)\le r < s\le p(z)$ such that $|\alpha(s)-\alpha(r)|> a-\delta$. Indeed, if $p(\tilde{z})=t\vee(p(z)-\mu)$, then we can choose $r=p(\tilde{z})$ and, given that $t\vee(p(z)-\mu) \notin \operatorname{Disc}(\alpha)$, this gives $r=p(\tilde{z})$. Otherwise, either $|\alpha(s)-\alpha(p(\tilde z))|> a-\delta$ or $|\alpha(s)-\alpha(p(\tilde z)-)|> a-\delta$ and the existence of such $r$ is guaranteed due to $\alpha$ being càdlàg. The desired contradiction follows, as in case (2), from $|\alpha(s)-u(\tilde{z})|> a-\delta$ for all $\delta>0$.
   (4) if $p(z) \in \operatorname{Disc}(\alpha)$ and $p(\tilde{z}) \notin \operatorname{Disc}(\alpha)$, a contradiction follows in a similar way to (3).
\end{proof}

\begin{lemma} \label{lem:prerequisite_2_suitable_sequence_of_large_increments_stopping_time_defi}
    The set $\{ a>0  :  \Pro^0\big( a \in V_{t,\mu}(H^0)\big)>0\}$
    is at most countable for every $t\ge 0$ and $\mu>0$.
\end{lemma}
\begin{proof}
    The process $a\mapsto \varsigma_{a,t,\mu}(H^0)$ is càdlàg by (i) of Lemma \ref{lem:prerequisite_suitable_sequence_of_large_increments_stopping_time_defi}. Note that $\{ a>0  :  \Pro^0\big( a \in V_{t,\mu}(H^0)\big)>0\}= \operatorname{Disc}_{\Pro^0}(a\mapsto \varsigma_{a,t,\mu}(H^0))$, and the latter set is at most countable as the set of discontinuities of a càdlàg process which occur with positive probability.
\end{proof}

\begin{lemma} \label{lem:prerequisite_final_suitable_sequence_of_large_increments_stopping_time_defi}
    Let $\Theta$ be a countable set with $|\Theta|=\operatorname{sup}_{x,y \in \Theta}|x-y| >0$. Then, the set 
    \begin{align}
        V^{\Theta} \; := \; \biggl\{ a>0  :  \Pro^0\Big( a \in \bigcup_{\nu\in \Theta} V_{\nu,\, |\Theta|}(H^0)\Big)>0 \biggr\} \label{eq:defi_set_V_Theta}
    \end{align} 
    is at most countable.
\end{lemma}
\begin{proof}
    Follows immediately from Lemma \ref{lem:prerequisite_2_suitable_sequence_of_large_increments_stopping_time_defi}.
\end{proof}

\begin{lemma} \label{lem:stopping_times_do_not_hit_the_partition_sets}
    Let $\ell \ge 1$ and $\Theta_\ell$ defined as in \eqref{eq:defi_set_V_Theta}, with with $|\Theta_\ell|=\operatorname{sup}_{x,y \in \Theta_\ell}|x-y| >0$. Then, the set 
    \begin{align}
        U^{\Theta_\ell} \; := \; \biggl\{ a>0  :  \Pro^0\Big( \bigcup_{\nu \in \Theta_\ell} \{ \varsigma_{a,\nu, |\Theta_\ell|}(H^0) \in \Theta_\ell\}\Big) >0 \biggr\} \label{eq:defi_set_V_Theta}
    \end{align} 
    is at most countable.
\end{lemma}
\begin{proof}
For $\nu,\eta\in\Theta_\ell$, define
\[
G_{\nu,\eta}
:=
\operatorname{max}_{1\le i\le d}\; \operatorname{sup}_{\nu\vee(\eta-|\Theta_\ell|)\le r\le\eta}
\, \bigl|H^{0,(i)}_r-H^{0,(i)}_\eta\bigr|.
\]
Since $\eta$ is almost surely a continuity point of $H^0$, we have
\[
\{\varsigma_{a,\nu,|\Theta_\ell|}(H^0)=\eta\}
\subseteq
\{G_{\nu,\eta}=a\}
\qquad \Pro^0\text{-a.s.}\,.
\]
Indeed, if $\varsigma_{a,\nu,|\Theta_\ell|}(H^0)=\eta$, then the defining oscillation is at most
$a$ before $\eta$ and exceeds $a$ arbitrarily close after $\eta$, so continuity at
$\eta$ yields $G_{\nu,\eta}=a$. Hence
\[
\bigl\{a>0:\Pro^0(\varsigma_{a,\nu,|\Theta_\ell|}(H^0)=\eta)>0\bigr\}
\subseteq
\bigl\{a>0:\Pro^0(G_{\nu,\eta}=a)>0\bigr\},
\]
and the latter set is at most countable, since a real-valued random variable has at
most countably many atoms. Finally,
\[
U^\Theta
=
\bigcup_{\nu,\eta\in\Theta_\ell}
\bigl\{a>0:\Pro^0(\varsigma_{a,\nu,|\Theta_\ell|}(H^0)=\eta)>0\bigr\},
\]
which is at most countable because $\Theta_\ell\times\Theta_\ell$ is countable.
\end{proof}

With the previous lemmas at hand, we can now prove Lemma \ref{lem:conv_relevant_quant_remainder_integrals}. Towards this proof, without loss of generality assume the $a_k$ to be chosen such that
$ a_k  \in  (0,\infty) \setminus ( V^{\Theta_\ell} \cup U^{\Theta_\ell} \cup W)$ for all $\ell \ge 1$, where $V^{\Theta_\ell}$ is defined in \eqref{eq:defi_set_V_Theta} of Lemma \ref{lem:prerequisite_final_suitable_sequence_of_large_increments_stopping_time_defi} and $W$ in Lemma \ref{lem:existence_of_suitable_sequence_not_hit_by_jumps}.

\begin{proof}[Proof of Lemma \ref{lem:conv_relevant_quant_remainder_integrals}]
Since we also have $(H^n,X^n)\Rightarrow (H^0,X^0)$, due to the uniqueness of weak limits and the choice of the $\Theta_\ell$ and $a_k$ made at the beginning of Section \ref{sec:proofs}, we recall that in particular
    $$ \Theta_\ell \subseteq  [0,\infty) \setminus \operatorname{Disc}_{\tilde \Pro^0}(\tilde{H}^0, \tilde{X}^0), \quad \tilde \Pro^0\Big( \bigcup_{\nu \in \Theta_\ell} \{ \varsigma_{a_k-,\nu,|\Theta_\ell|}(\tilde{H}^0) < \varsigma_{a_k,\nu,|\Theta_\ell|}(\tilde{H}^0)\} \Big)=0, $$
    $$ \text{and } \; \tilde \Pro^0\Big( \bigcup_{\nu \in \Theta_\ell} \{ \varsigma_{a_k,\nu, |\Theta_\ell|}(\tilde H^0) \in \Theta_\ell\}\Big) =0,$$
    as well as $|\Theta_\ell| \notin \{|x-y|:x,y \in \operatorname{Disc}_{\tilde \Pro^0}(\tilde{H}^0, \tilde{X}^0)\}$, for all $\ell,k \ge 1$, where the quantities $\varsigma_{a_k,\nu,|\Theta_\ell|}$ are defined by
    \begin{align*}
    \varsigma_{a_k,\nu,|\Theta_\ell|}(\alpha)  := \operatorname{inf}\{s>\nu  : \operatorname{sup} \{ |\alpha^{(i)}(r)-\alpha^{(i)}(s)|  :  t \vee (
    s- |\Theta_\ell|) \le r \le s,  1\le i \le d\} > a_k \}.
\end{align*}
    as in \eqref{eq:defi_stopping_times_for_appendix_lemma_on_defi_of_suitable_sequence} in the Appendix.
    By the generalised continuous mapping theorem \cite[Thm.~2.7]{billingsley}, it suffices to prove convergence for the deterministic counterparts of the quantities in \eqref{eq:conv_relevant_quant_remainder_integrals} on $(\D_{\R^d}[0,\infty),\dM)^2$. In alignment with this, fix $k,\ell\ge 1$ and assume $h_n:=H^n$, $x_n:=X^n$ and $h_0:=\tilde{H}^0$, $x_0:=\tilde{X}^0$ to be deterministic càdlàg paths such that $\dM(h_n,h_0)\vee \dM(x_n,x_0)\to 0$, as well as $\Theta_\ell \notin \operatorname{Disc}(h_0,x_0)$, $|\Theta_\ell| \notin \{|x-y|:x,y \in \operatorname{Disc}(h_0,x_0)\}$, $\varsigma_{a_k-,\nu,|\Theta_\ell|}(h_0) = \varsigma_{a_k,\nu,|\Theta_\ell|}(h_0)$ for all $\nu \in \Theta_\ell$, and $\varsigma_{a_k,\nu,|\Theta_\ell|}(h_0) \notin \Theta_\ell$. In particular, the latter implies that $\tau^{0,k,\ell}_j \notin \Theta_\ell$ for all $k,\ell,j\ge 1$, by its definition. In order to ease notation, we will omit the indices $k,\ell$ for all respective quantities. This will specifically apply to the $\tau^{n,k,\ell}_j$, $\Theta_\ell$ and $a_k$.
    \begin{enumerate}[(1)]
       \item \label{it1:conv_rel_quant_remainder_int} As a first step, we note that $\tau^n_1 \to \tau^0_1$. Indeed, $\tau^n_1=\varsigma_{a_k,0,|\Theta_\ell|}(h_n)$ for all $n\ge 0$, $0 \in \Theta$ and the continuity result in Lemma \ref{lem:prerequisite_suitable_sequence_of_large_increments_stopping_time_defi}(iv) yield $\tau^n_1 \to \tau^0_1$.
       \item \label{it2:conv_rel_quant_remainder_int} Recall the definition of $\lceil \tau_j^n\rceil_\Theta$, $\lfloor \tau_j^n\rfloor_\Theta$ and $|\Theta|^{\downarrow}$ from the beginning of Section \ref{sec:proofs}. Suppose $\tau^n_j\to \tau^0_j$ for some $j\ge 1$. Then, $\lceil \tau_j^n\rceil_\Theta = \lceil \tau^0_j\rceil_\Theta$ and $\lfloor \tau_j^n\rfloor_\Theta = \lfloor \tau^0_j\rfloor_\Theta$ for large enough $n$. To see this, first recall that $\tau^0_j \notin \Theta$. Let $\nu_1, \nu_2 \in \Theta$ such that $\nu_1 < \tau^0_j < \nu_2$ and $(\nu_1, \nu_2) \cap \Theta=\emptyset$, implying $\lceil \tau^0_j\rceil_\Theta=\nu_2$, $\lfloor \tau^0_j\rfloor_\Theta=\nu_1$. Then, $\tau^n_j \in (\nu_1, \nu_2)$ for large enough $n$ and we immediately obtain $\lceil \tau^n_j\rceil_\Theta=\nu_2$, $\lfloor \tau^n_j\rfloor_\Theta=\nu_1$.
       \item \label{it3:conv_rel_quant_remainder_int} Suppose  $\lceil \tau_j^n\rceil_\Theta = \lceil \tau^0_j\rceil_\Theta$ for large enough $n$. This implies $\tau^0_{j+1}=\varsigma_{a_k,\lceil \tau^0_j\rceil_\Theta,|\Theta_\ell|}(h_0)$, $\tau^n_{j+1}=\varsigma_{a_k,\lceil \tau^0_j\rceil_\Theta,|\Theta_\ell|}(h_n)$, $\lceil \tau^0_j\rceil_\Theta \in \Theta$, for large enough $n$. As in \eqref{it1:conv_rel_quant_remainder_int}, we obtain $\tau^{n}_{j+1} \to \tau^0_{j+1}$ by Lemma \ref{lem:prerequisite_suitable_sequence_of_large_increments_stopping_time_defi}(iv). Induction from \eqref{it1:conv_rel_quant_remainder_int} and \eqref{it2:conv_rel_quant_remainder_int} finally yields $\tau^n_j\to \tau^0_j$, $\ind_{\{\lceil \tau_j^n\rceil_\Theta = \lceil \tau^0_j\rceil_\Theta\}}\to 1$ and $\ind_{\{\lfloor \tau_j^n\rfloor_\Theta =\lfloor \tau^0_j\rfloor_\Theta\}} \to 1$ as $n\to \infty$ for all $j\ge 1$.
       \item \label{it4:conv_rel_quant_remainder_int} We have that $\dM(h_n,h_0) \to 0$ as $n\to \infty$ implies $h_n(t) \to h_0(t)$ for all $t \notin \operatorname{Disc}(h_0)$ (see e.g. \cite[Lem.~12.5.1]{whitt}). Since $\Theta \subseteq [0,\infty) \setminus \operatorname{Disc}(h_0)$ and $\lceil \tau^0_j\rceil_\Theta, \lfloor \tau^0_j\rfloor_\Theta \in \Theta$ for all $j\ge 1$, we deduce $h_n( \lceil \tau^0_j\rceil_\Theta) \to h_0( \lceil \tau^0_j\rceil_\Theta)$, $h_n( \lfloor \tau^0_j\rfloor_\Theta) \to h_0( \lfloor \tau^0_j\rfloor_\Theta)$ for all $j\ge 1$. Due to \eqref{it3:conv_rel_quant_remainder_int}, $\lceil \tau^n_j\rceil_\Theta=\lceil \tau^0_j\rceil_\Theta$ and $\lfloor \tau^n_j\rfloor_\Theta = \lfloor \tau^0_j\rfloor_\Theta$ for large enough $n$, therefore implying $h_n(\lceil\tau^n_j\rceil_\Theta) \to h_0(\lceil\tau^0_j \rceil_\Theta)$ and $h_n(\lfloor\tau^n_j\rfloor_\Theta) \to h_0(\lfloor \tau^0_j \rfloor_\Theta)$ as $n\to \infty$, for every $j\ge 1$. Finally, apply the same argument to $x_n$, $x_0$.
   \end{enumerate}
\end{proof}

\noindent \textbf{A lemma for the proof of Theorem \ref{cor:main_result_M1_J1_case}}
\begin{lemma} \label{lem:joint_convergence_from_coinciding_jump_times}
    Let $\{Z^n\}$, $\{\tilde{Z}^n\}$, $\{H^n\}$ and $\{X^n\}$ be stochastic processes such that there is weak convergence $(H^n,X^n,Z^n,\tilde{Z}^n)\Rightarrow (H,X,Z,\tilde{Z})$ on the space $(\D_{\R^d}[0,\infty), \dM)\times (\D_{\R^d}[0,\infty), \dJ)^3$. If, for every $n\ge 1$, it holds $|\Delta Z^n_s| \vee |\Delta \tilde{Z}^n_s|\le 2|H^n|^*_s|\Delta X^n_s|$ almost surely for all $s>0$. Then, $X^n$, $Z^n$ and $\tilde{Z}^n$ converge together weakly on the strong $J_1$ Skorokhod space, that is
    $$ (X^n,Z^n,\tilde{Z}^n)\Rightarrow (X,Z,\tilde{Z}) \qquad \text{on} \quad (\D_{\R^{3d}}[0,\infty), \dJ). $$
\end{lemma}
\begin{proof}
    Denote $\Lambda_\alpha:=\{(h,x,z,\tilde{z}) : |\Delta z_s|\vee|\Delta \tilde{z}_s|\le 2|h|^*_s|\Delta x_s| \text{ for all } s>0 \} \subseteq (\D_{\R^d}[0,\infty))^4$ and write $\rho(\lambda_1,...,\lambda_4),(\beta_1,...,\beta_4)):= \dM(\lambda_1,\beta_1) \vee \operatorname{max}\{\dJ(\lambda_i,\beta_i):i=2,...,4\}$. If we can show that the map 
    $$ g:\; (\Lambda_\alpha, \rho) \to (\D_{\R^{3d}}[0,\infty),\dJ), \;\;  (h,x,z,\tilde{z}) \mapsto (x,z,\tilde{z})$$
    is sequentially continuous in the sense $g(y_n)\to g(y)$ whenever $y_n \in \Lambda_\alpha$ and $y_n\to y \in \overline{\Lambda_\alpha}$, then the result of the lemma follows from the generalised continuous mapping theorem \cite[Thm.~1.11.1]{vandervaart}. To this aim, let $\{(h_n,x_n,z_n, \tilde{z}_n)\} \subseteq \Lambda_\alpha$ converging to some $(h,x,z, \tilde{z}) \in \overline{\Lambda_\alpha}$ with respect to $\rho$ as $n \to \infty$. Fix $T>0$ a continuity point of $(z,\tilde z, h, x)$ and recall that $|h_n|^*_T$ is uniformly bounded by some $K>1$ as a result of the sequence of $h_n$ being relatively compact in $M_1$. Denote by $\lambda^{z}_n,\lambda^{\tilde z}_n,\lambda^{x}_n:[0,T] \to [0,T]$ the increasing homeomorphisms such that $|x_n\circ \lambda^x_n -x|^*_T \vee |\lambda^x_n-\Id|^*_T \to 0$, and analogously for $z_n, \tilde{z}_n$. 

    Now, for any $s \in \operatorname{Disc}(x)\cap \operatorname{Disc}(z)= \operatorname{Disc}(z)$, there is $N\ge 1$ such that $\lambda^{z}_n(s)=\lambda^x_n(s)$ for all $n\ge N$. Indeed, otherwise there exist $\varepsilon>0$, $s \in \operatorname{Disc}(x)\cap \operatorname{Disc}(z)$ with $|\Delta x(s)|\vee |\Delta z(s)|>\varepsilon$, and a subsequence satisfying $\lambda_n^z(s)\neq \lambda_n^x(s)$ for all $n\ge 1$. Recall that $\lambda^x_n(s), \lambda^z_n(s)\to s$ as $n\to \infty$. It is straightforward to show $|\Delta x_n(\lambda_n^x(s))| \to |\Delta x(s)|>\varepsilon$ and $|\Delta z_n(\lambda_n^z(s))| \to |\Delta z(s)|>\varepsilon$. Since $|\Delta z_n(\lambda^z_n(s))|\le 2|h_n|^*_{\lambda^z_n(s)}|\Delta x_n(\lambda^z_n(s))|\le 2K |\Delta x_n(\lambda^z_n(s))|$, we deduce $|\Delta x_n(\lambda^z_n(s))|> \varepsilon/K$ for large enough $n$. However, this means that on any interval $(s-\theta,s+\theta)$, $\theta>0$, the $x_n$ jump twice with jump sizes at least $\varepsilon/K$ for large enough $n$ (precisely at $\lambda^z_n(s), \lambda^x_n(s))$, and therefore cannot converge in $J_1$, which yields a contradiction. In full analogy, we obtain that for any $s \in \operatorname{Disc}(x)\cap \operatorname{Disc}(\tilde z)=\operatorname{Disc}(\tilde z)$ there is $N\ge 1$ such that $\lambda^{\tilde z}_n(s)=\lambda^x_n(s)$ for all $n\ge N$.

    In a next step, observe that $|z_n\circ \lambda^x_n - z|^*_T \to 0$ is equivalent to $|z_n\circ \lambda^z_n - z_n\circ \lambda^x_n |^*_T\to 0$. Assume towards a contradiction that $|z_n\circ \lambda^z_n - z_n\circ \lambda^x_n |^*_T \not \to 0$. Then, there exist $\varepsilon>0$, a subsequence $\{n_k\}$ and times $\{s_{k}\} \subseteq [0,T]$ with $|z_{n_k}(\lambda^z_{n_k}(s_k)) - z_{n_k}(\lambda^x_{n_k}(s_k)) |>\varepsilon$ for all $k\ge 1$. Extract a further subsequence of times, which we will also denote by $\{s_k\}$ such that $s_k \to s \in [0,T]$. Assume that $\{s_k\}$ is increasing (the case $\{s_k\}$ decreasing can be treated similarly). Then, we obtain
    \begin{align}
        \varepsilon \, &< \, |z_{n_k}(\lambda^z_{n_k}(s_k)) - z_{n_k}(\lambda^x_{n_k}(s_k)) | \notag \\ 
        &\le \;  |z_{n_k}\circ \lambda^z_{n_k}-z|^*_T  \; + \; |z(s_k)- z\circ (\lambda^z_{n_k})^{-1}(\lambda^x_{n_k}(s_k))| \; + \; |z_{n_k}-z\circ (\lambda^z_{n_k})^{-1}|^*_T \label{eq:lem_aux_joint_convergence} 
    \end{align}
    for each $k\ge 1$, and, if $s \notin \operatorname{Disc}(z)$, \eqref{eq:lem_aux_joint_convergence} immediately leads to a contradiction due to $|(\lambda^z_{n})^{-1}-\Id|^*_T\vee |\lambda^x_{n}-\Id|^*_T \to 0$. On the other hand, if $s \in \operatorname{Disc}(z)$, then it is known from above that $\lambda^z_{n}(s)=\lambda^x_n(s)$ for large enough $n$ and thus, since both homeomorphisms are strictly increasing, $\lambda^z_{n_k}(s_k), \lambda^x_{n_k}(s_k) < \lambda^z_{n_k}(s)=\lambda^x_{n_k}(s)$ for large enough $k$. Consequently, $s>(\lambda^z_{n_k})^{-1}(\lambda^x_{n_k}(s_k))$ for large enough $k$ and at the same time we still have $(\lambda^z_{n_k})^{-1}(\lambda^x_{n_k}(s_k)) \to s$. Hence, we deduce a contradiction to \eqref{eq:lem_aux_joint_convergence} since $z(s_k) \to z(s-)$ and $z\circ (\lambda^z_{n_k})^{-1}(\lambda^x_{n_k}(s_k)) \to z(s-)$. A similar procedure shows $|\tilde z_n\circ \lambda^x_n-\tilde z|^*_T\to 0$. 

    Finally, this yields $|(x_n,z_n,\tilde z_n) \circ \lambda^x_n - (x,z,\tilde z)|^*_T \to 0$ and therefore the convergence $(z_n,\tilde z_n, x_n) \to (z,\tilde z, x)$ on $(\D_{\R^{3d}}[0,\infty), \dJ)$.
\end{proof}




\textbf{Acknowledgements.}
I would like to thank an anonymous referee not only for offering a plethora of suggestions to improve clarity and readability of this inherently technical paper, but also for a pivotal remark which led to the inclusion of the conjecture at the end of Section \ref{sec:main_results_M1_M1}.


\end{document}